\documentclass[a4paper]{article}
\usepackage{amsmath}
\usepackage{amsthm}
\usepackage{amssymb}
\usepackage{amscd}
\usepackage{amsfonts}
\usepackage{amsbsy}

\usepackage{graphicx}
\usepackage{authblk}
\usepackage{hyperref}
\usepackage{multirow}

\usepackage{cite}
\usepackage{comment}

\usepackage{epsfig,afterpage}
\usepackage[dvips]{psfrag}
\usepackage{subfigure}
\usepackage{srcltx}
\usepackage{xcolor}

\usepackage{verbatim}

\usepackage{graphics}
\usepackage{epstopdf}
\usepackage{overpic}
\usepackage{multicol}

\newtheorem{theorem}{Theorem}[section]
\newtheorem{proposition}[theorem]{Proposition}
\newtheorem{lemma}[theorem]{Lemma}

\newtheorem{remark}{Remark}

\newcommand{\includegraph}[2][]{\ifnum\pdfoutput=0\includegraphics[#1]{#2.eps}\else\includegraphics[#1]{#2.pdf}\fi}

\newtheorem{mtheorem}{Theorem}

\author[1]{Renato Huzak}
\author[2]{Otavio Henrique Perez}

\affil[1]{Hasselt University, Campus Diepenbeek, Agoralaan Gebouw D, 3590 Diepenbeek, Belgium}
\affil[2]{Universidade de S\~{a}o Paulo (USP), Instituto de Ci\^{e}ncias Matem\'aticas e de Computa\c{c}\~{a}o (ICMC). Avenida Trabalhador S\~{a}o Carlense, 400, CEP 13566-590, S\~{a}o Carlos, S\~{a}o Paulo, Brazil}

\title{An unbounded number of canard limit cycles in linear regularizations of piecewise linear systems\footnote{Corresponding Author: Otavio Henrique Perez}}

\date{}

\begin{document}
\maketitle

\begin{abstract}

The purpose of this paper is to study the number of limit cycles of canard type in linear regularizations of piecewise linear systems with non-monotonic transition functions. Using the notion of slow divergence integral and elementary breaking mechanisms, we construct systems with an arbitrary finite number of hyperbolic limit cycles. The Hopf breaking mechanism deals with transition functions with precisely one critical point in the interval $(-1,1)$. On the other hand, the jump breaking mechanism produces any number of limit cycles using transition functions with precisely three critical points in $(-1,1)$.  

\end{abstract}
\textit{Keywords:} Canard cycles; slow divergence integral; slow-fast Hopf point; jump point; regularization. \newline
\textit{2020 Mathematics Subject Classification:} 34E15, 34E17, 34C40.

\tableofcontents

\section{Introduction}\noindent

In this paper we consider planar piecewise linear (PWL) systems of the form
\begin{align}
 \dot z &=\begin{cases}
           X(z)  \text{ for }h(z)>0,\\
           Y(z)  \text{ for }h(z)<0,
          \end{cases}\label{pws}\quad z=(x,y)\in \mathbb{R}^{2},
\end{align}
where the vector fields $X=(X_1,X_2)$ and $Y=(Y_1,Y_2)$ and the function $h:\mathbb R^2\rightarrow \mathbb R$, $\nabla h\ne 0$, are each affine. The set $\Sigma:=h^{-1}(0)$ is called the switching line. This class of piecewise smooth vector fields is the central topic of a wide number of papers, see for instance \cite{Gasull2020,LlibreOrd,Han2010,LT,FrPoTo,HuanYang2,LiLiuLli,LiLli,MedTorr}.

Limit cycles of \eqref{pws} have attracted great attention from many mathematicians and interesting results have been proven. Lum and Chua \cite{lum1991a} conjectured that the number of limit cycles of  \eqref{pws} in the continuous case (i.e., $X(z)=Y(z)$ for all $z\in \Sigma$) is at most one. The conjecture was proven by Freire et al. \cite{Freire}.

The determination of the maximum number of crossing limit cycles in discontinuous PWL systems \eqref{pws} is more challenging (see e.g. \cite{BragaMello,HuanYang,freire2013a,llibre2013a,Llibre3LC,{esteban2021a,li2021a}} and references therein). Using a case-independent approach based on integral representations of the Poincar\'e half-maps \cite{Carmona}, Carmona, Fern\'andez-S\'anchez and Novaes showed that the maximum number of crossing limit cycles is uniformly bounded by $8$ (see \cite{carmona2023a}) and gave the first case-independent proof of Lum and Chua's conjecture (see \cite{carmona2021a}). To the best of our knowledge, $3$ crossing limit cycles have been found (see \cite{HuanYang,Llibre3LC}).  

We point out that the interest in the number of crossing limit cycles for \eqref{pws} is closely related to the second part of Hilbert's 16th problem {\cite{smale}}. The problem asks if there is a finite upper bound on the number of limit cycles for polynomial vector fields of a given degree $n$, and it is unsolved even for quadratic vector fields \cite{Program}.

In this paper, we focus on the following natural question, also related to Hilbert's 16th problem: \textit{Is there an upper bound on the number of limit cycles of regularized piecewise polynomial vector fields}? Even though regularized vector fields are not polynomial, this question is still interesting and its answer is non-trivial as we shall discuss. 

In \cite{RHKK}, it has been proved that the number of limit cycles of regularized (according to Sotomayor--Teixeira \cite{Sotomayor96}) piecewise quadratic systems is unbounded. More precisely, there exists a piecewise quadratic vector field satisfying the following property: for a given integer $k > 0$, there is a monotonic transition function $\varphi_{k}$ such that the regularized vector field has at least $k+1$ hyperbolic limit cycles. Of course, one can also expect an unbounded number of limit cycles when considering Sotomayor--Teixeira regularized piecewise polynomial systems of higher degree. We believe that the only case where one could expect a finite upper bound on the number of limit cycles is Sotomayor--Teixeira regularizations of PWL systems.

One can also state a similar problem for different regularization processes. In \cite{Otavio25} the authors considered the so called \textit{non linear regularizations} and proved that non-linearly regularized PWL vector fields can also produce an unbounded number of limit cycles using monotonic transition functions. This result is true even for non linear regularizations of quadratic degree, and the very same paper considers limit cycles of non linearly regularized PWL vector fields of higher degree. See \cite{Jeff,NovJeff} for theoretical aspects and applications of such regularizations.

The goal of this paper is to prove that the number of limit cycles of \textit{linear regularizations} of PWL systems \eqref{pws} is \textit{unbounded}, using \textit{non monotonic} transition functions. It will be clear in the next sections that linear regularizations with monotonic transition functions coincide with the classical Sotomayor--Teixeira regularization. Therefore, it remains an open problem to find the maximum number of limit cycles in regularized PWL vector fields following the classical Sotomayor--Teixeira process.

More precisely, we consider the $\varepsilon$-family of smooth vector fields  
\begin{equation}
    \label{LINREG-Intro}
    Z_\varepsilon(x,y):=\frac{1+\varphi\left(\frac{h(x,y)}{\varepsilon}\right)}{2}X(x,y)+\frac{1-\varphi\left(\frac{h(x,y)}{\varepsilon}\right)}{2}Y(x,y),
\end{equation}
 where $\varepsilon>0$ is a small parameter, $X$, $Y$ and $h$ are introduced in \eqref{pws}. Here, we assume that $h(x,y)=x$ and $\varphi:\mathbb{R}\rightarrow\mathbb{R}$ satisfies the following conditions: \textbf{(1)} $\varphi$ is $C^{\infty}$-smooth and \textbf{(2)} $\varphi(t) = -1$ if $t \leq -1$ and $\varphi(t) = 1$ if $t \geq 1$. We call $\varphi$ a transition function. We say that the transition function $\varphi$ is monotonic if it satisfies
\textbf{(3)} $\varphi'(t) > 0$ if $t\in (-1,1)$. The system \eqref{LINREG-Intro} is called a {$\varphi$-linear regularization of \eqref{pws} and, if $\varphi$ is monotonic, then \eqref{LINREG-Intro} is the well-known Sotomayor--Teixeira regularization \cite{Sotomayor96}.

We show that there exist linear vector fields $X$ and $Y$ such that the following is true: for any integer $k > 0$, there exists a non-monotonic transition function $\varphi :\mathbb{R}\rightarrow\mathbb{R}$ such that the {$\varphi$-linear regularization  \eqref{LINREG-Intro} has at least $k+1$ hyperbolic limit cycles, for each $\varepsilon>0$ small enough. Moreover, the number of critical points of $\varphi$ is fixed (i.e., independent of $k$). For a precise statement of this result we refer to Theorems \ref{thm-main-Hopf} and \ref{thm-main-jump} in Section \ref{section-statements}.

The hyperbolic limit cycles in Theorems \ref{thm-main-Hopf} and \ref{thm-main-jump} are produced by canard cycles associated with the slow-fast system 
\begin{equation}\label{LINREG-rescaled}
\left\{
\begin{array}{rcl}
   \dot{x} &=& \displaystyle\frac{X_{1}(\varepsilon x,y) + Y_{1}(\varepsilon x,y)}{2} + \varphi(x)\frac{X_{1}(\varepsilon x,y) - Y_{1}(\varepsilon x,y)}{2}, \\ 
   \dot{y} &= &\varepsilon\left(\displaystyle\frac{X_{2}(\varepsilon x,y) + Y_{2}(\varepsilon x,y)}{2} + \varphi(x)\frac{X_{2}(\varepsilon x,y) - Y_{2}(\varepsilon x,y)}{2}\right),
\end{array}
\right.
\end{equation}
obtained after the performing $x=\varepsilon\tilde{x}$ to \eqref{LINREG-Intro} and multiplication by $\varepsilon>0$ (we drop the tildas in \eqref{LINREG-rescaled}). Canard cycles are limit periodic sets of \eqref{LINREG-rescaled}, defined for $\varepsilon=0$, that can produce limit cycles of \eqref{LINREG-rescaled} (or \eqref{LINREG-Intro}) for $\varepsilon>0$ small. They consist of fast (horizontal) orbits and at least one attracting and one repelling portion of the curve of singularities (for more details, see Section \ref{section-def-Hopf-jump}). Clearly, we have to find appropriate PWL vector field $(X,Y)$ and transition function $\varphi$ in \eqref{LINREG-rescaled} such that the canard cycles exist. 

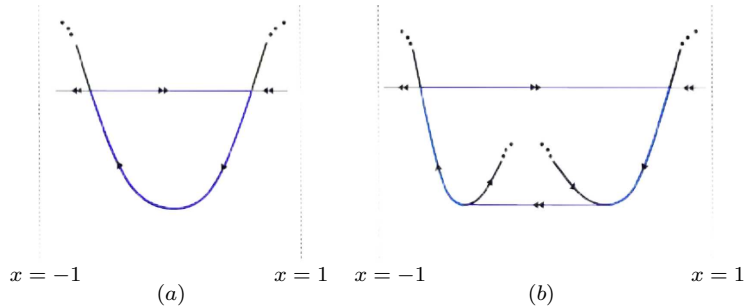
\begin{figure}[ht]
	\begin{center}
		\includegraphics[width=8.9cm,height=3.4cm]{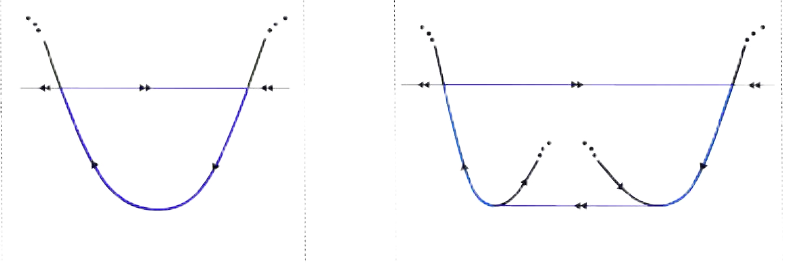}
		{\footnotesize 
        \put(-209,-14){$(a)$}
        \put(-70,-14){$(b)$}
                \put(-264,-7){$x=-1$}
                        \put(-165,-7){$x=1$}
                        \put(-135,-7){$x=-1$}
                        \put(-9,-7){$x=1$}
}
         \end{center}
	\caption{Possible canard cycles in \eqref{LINREG-rescaled} for $\varepsilon=0$. (a) Canard cycle with a slow-fast Hopf point. (b)  Canard cycle containing a jump connection.}
	\label{fig-both-mechanisms} 
\end{figure}

In this paper, we consider two important types of canard cycles that can occur in \eqref{LINREG-rescaled} when $\varepsilon\to 0$. The first type is related to the so-called \textit{Hopf breaking mechanism} \cite[Section 6.3]{DDR-book-SF} (see also Section \ref{subsection-Hopf}). This mechanism contains a slow-fast Hopf/canard point near which the passage from an attracting branch to a repelling branch of the curve of singularities is possible, see Figure \ref{fig-both-mechanisms}(a). If we add a breaking parameter to such a slow-fast Hopf point, canard cycles  in Figure \ref{fig-both-mechanisms}(a) can produce limit cycles of \eqref{LINREG-rescaled} for $\varepsilon> 0$ small. The Hopf breaking mechanism has been used to prove Theorem \ref{thm-main-Hopf}, with the breaking parameter included in $X$ and $Y$ (see Section \ref{section-proof-Hopf-main}).

The second type deals with the so-called \textit{jump mechanism}. In this case, we have two (generic) jump points that are connected by a fast orbit, see Figure \ref{fig-both-mechanisms}(b). If we add a breaking parameter to such a connection, we get a (generic) jump breaking mechanism  \cite[Section 6.2]{DDR-book-SF} (see also Section \ref{subsection-jump}). In Theorem \ref{thm-main-jump}, we generate limit cycles using the jump breaking mechanism, where we include the breaking parameter in $\varphi$ (see Section \ref{section-proof-jump-main}).

If system \eqref{LINREG-rescaled} has a slow-fast Hopf point or a jump point at $(x,y)=(x_{0}, y_{0})$, with $x_{0}\in (-1,1)$, then $\varphi'(x_{0})=0$ (that is, $x_{0}$ is a critical point of the transition function), see Section \ref{section-def-Hopf-jump}. This has already been observed in \cite{Otavio23} for jump points. We can therefore have a Hopf (or jump) breaking mechanism in the slow-fast system \eqref{LINREG-rescaled} only if we drop the monotonicity condition in the Sotomayor--Teixeira regularization. This paper can also be seen as a continuation of \cite{Otavio23}, when we deal with limit cycles produced by canard cycles such as in Figure \ref{fig-both-mechanisms}(b). For a precise definition of slow-fast Hopf points and jump points, see Section \ref{section-def-Hopf-jump}.

In Theorems \ref{thm-main-Hopf} and \ref{thm-main-jump}, we choose $X$ and $Y$ in such a way that the slow-fast system \eqref{LINREG-rescaled} becomes a classical Li\'enard equation (see e.g. the system \eqref{classical} in Section \ref{section-def-Hopf-jump}). Limit cycles of canard type in (slow-fast) classical Li\'enard equations can be studied using the notion of slow divergence integral \cite[Chapter 5]{DDR-book-SF} (see \cite{DDMoreLC,SDICLE1,DPR,DHGener} and references therein). Simple zeros of the slow divergence integral correspond to hyperbolic limit cycles (see Theorems \ref{thm-Hopf-simple-zeros} and \ref{thm-Jump-simple-zeros} in Sections \ref{subsection-Hopf} and \ref{subsection-jump}). Moreover, the main tool applied in the proof of the main results of \cite{Otavio25,RHKK} discussed previously was the slow divergence integral.

In this paper, we are only interested in the limit cycles produced by canard cycles located inside the regularization stripe. The study of canard cycles with portions located outside the stripe, in linear regularizations of PWL systems with non-monotonic transition functions, is left for future research. See also \cite{Otavio25}.

This paper is organized as follows. In Section \ref{section-statements}, we state our main results. In Section \ref{section-def-Hopf-jump}, we define Hopf and jump mechanisms in slow-fast systems, and then we discuss these notions for classical Li\'enard equations and regularizations of PWL systems. Finally, in Section \ref{section-proof-all}, we prove Theorems \ref{thm-main-Hopf} and \ref{thm-main-jump}.


\section{Statement of the main results}\label{section-statements}\noindent

We consider $h(x,y) = x$ in Equation \eqref{LINREG-Intro}, and we obtain
\begin{equation}
    \label{LINREG}
    Z_\varepsilon(x,y):=\frac{1+\varphi\left(\frac{x}{\varepsilon}\right)}{2}X(x,y)+\frac{1-\varphi\left(\frac{x}{\varepsilon}\right)}{2}Y(x,y),
\end{equation}
where $\varepsilon>0$ is a parameter kept small, $X=(X_1,X_2)$ and $Y=(Y_1,Y_2)$. Theorems \ref{thm-main-Hopf} and \ref{thm-main-jump} stated below deal with the number of limit cycles of \eqref{LINREG}, with linear vector fields $X$ and $Y$ and non-monotonic transition functions $\varphi$. 

\begin{mtheorem}[Hopf breaking mechanism]
    \label{thm-main-Hopf} 
    There exist linear vector fields $X(\cdot,\alpha)$ and $Y(\cdot,\alpha)$, depending on a parameter $\alpha\in \mathbb R$, such that the following is true: for any integer $k>0$, there exist a non-monotonic transition function $\varphi_k:\mathbb{R}\rightarrow\mathbb{R}$, with precisely $1$ critical point in $(-1,1)$, and a smooth
function $a_k:[0,\varepsilon_k]\to\mathbb R$, with $\varepsilon_k>0$ small enough and $a_k(0)=0$, such that the $\varphi_k$-linear regularization in \eqref{LINREG} with $\alpha=\varepsilon a_k(\varepsilon)$ has at least $k+1$ hyperbolic limit cycles, for each $\varepsilon\in(0,\varepsilon_k]$.
\end{mtheorem}

We prove Theorem \ref{thm-main-Hopf} in Section \ref{section-proof-Hopf-main}. The limit cycles in this theorem are produced using the Hopf breaking mechanism (see Figures \ref{fig-both-mechanisms}(a) and \ref{fig-Hopf}). On the other hand, Theorem \ref{thm-main-jump} (proven in Section \ref{section-proof-jump-main}) deals with limit cycles generated by generic jump breaking mechanisms (see Figures \ref{fig-both-mechanisms}(b) and \ref{fig-jump}).

\begin{mtheorem}[Jump breaking mechanism]
    \label{thm-main-jump}
There exist linear vector fields $X$ and $Y$ such that the following is true: for any integer $k>0$, there exist a smooth $b$-family of non-monotonic transition functions $\varphi_{b,k}:\mathbb{R}\rightarrow\mathbb{R}$, with precisely $3$ critical points in $(-1,1)$, and a continuous function $b_k:[0,\varepsilon_k]\to\mathbb R$, with $\varepsilon_k>0$ small enough and $b_k(0)=0$, such that the $\varphi_{b,k}$-linear regularization in \eqref{LINREG} with $b=b_k(\varepsilon)$ has at least $k+1$ hyperbolic limit cycles, for each $\varepsilon\in(0,\varepsilon_k]$.
\end{mtheorem}

The critical point of $\varphi_k$ in Theorem \ref{thm-main-Hopf} is of Morse type. The same is true for the critical points of $\varphi_{b,k}$ in Theorem \ref{thm-main-jump}. Notice that in Theorems \ref{thm-main-Hopf} and \ref{thm-main-jump} the number of critical points is fixed (it does not increase as $k$ increases).

\section{Generic breaking mechanisms}\label{section-def-Hopf-jump}\noindent

In this section, we consider planar slow-fast systems of the form
\begin{equation}\label{eq-def-slowfast-1}
\left\{
\begin{array}{rcl}
\dot{x} & = & f(x,y,\varepsilon), \\
\dot{y} & = & \varepsilon^{l} g(x,y,\varepsilon),
\end{array}
\right.
\end{equation}
where $\varepsilon\ge 0$ is the singular perturbation parameter kept close to zero, $l$ is a positive integer, and $f$ and $g$ are $C^\infty$-smooth functions. The overdot denotes the derivative of $x(t)$ and $y(t)$ with respect to the \textit{fast time} $t$. When $\varepsilon = 0$, the set $S = \{f(x,y,0) = 0\}$ is a curve of singularities of \eqref{eq-def-slowfast-1}, and it has horizontal intervals as fast regular orbits. A singularity $(x_{0},y_{0})\in S$ is \textit{normally hyperbolic} if $\frac{\partial f}{\partial x}(x_{0},y_{0},0) \neq 0$.

We also recall the notion of \textit{classical Li\'enard equations}, because they will play an important role in the proofs of Theorems \ref{thm-main-Hopf} and \ref{thm-main-jump}. Consider a classical Li\'enard equation 
\begin{equation}\label{classical}
    \begin{cases}
        \dot{x} = y-F(x), \\
        \dot{y} = -\varepsilon^{2} x,
    \end{cases}
    \end{equation}
where $F$ is a $C^\infty$-smooth function, and we adopt $l = 2$ (we postpone the explanation of this choice to Section \ref{sec-reg-mech}). When $\varepsilon=0$, the set $S = \{y = F(x)\}$ is a curve of singularities of system \eqref{classical}, and a singularity $(x_{0},F(x_{0}))\in S$ is normally hyperbolic if $F'(x_{0})\neq 0$ (the prime denotes the derivative with respect to $x$). The singularity is attracting when $F'(x_{0}) > 0$ and repelling when $F'(x_{0}) < 0$. Singularities $(x_{0},F(x_{0}))\in S$ that satisfy $F'(x_{0}) = 0$ are called \textit{contact points}.  

Limit cycles of slow-fast systems \eqref{eq-def-slowfast-1} (and in particular of classical Li\'enard systems \eqref{classical}) can be produced by slow-fast cycles defined for $\varepsilon = 0$. A \textit{slow-fast cycle} consists of fast orbits and compact portions of $S$. We say that a slow–fast cycle is a \textit{canard cycle} if it contains at least one attracting and one repelling portion of $S$. When canard cycles appear, the following two mechanisms play an essential role: \textbf{(a)} Hopf mechanism \cite[Section 6.3]{DDR-book-SF} (see Section \ref{subsection-Hopf}) and \textbf{(b)} jump mechanism \cite[Section 6.2]{DDR-book-SF} (see Section \ref{subsection-jump}).

\subsection{The Hopf breaking mechanism}\label{subsection-Hopf}\noindent

A singularity $(x_{0}, y_{0})\in S$ is a \textit{slow-fast Hopf point of \eqref{eq-def-slowfast-1}} if (see also \cite[Definition 2.4]{DDR-book-SF}) 
\begin{equation}\label{eq-def-hopf-turning-point}
\begin{split}
    f(x_{0}, y_{0},0) = g(x_{0}, y_{0},0) = \frac{\partial f}{\partial x}(x_{0}, y_{0},0) = 0, \qquad \  \\ 
    \frac{\partial^{2} f}{\partial x^{2}}(x_{0}, y_{0},0) \neq 0 \  \text{ and } \ \frac{\partial g}{\partial x}(x_{0}, y_{0},0)\cdot \frac{\partial f}{\partial y}(x_{0}, y_{0},0) < 0, \
\end{split}
\end{equation}

Assume that \eqref{eq-def-slowfast-1}
has a slow-fast Hopf point at $(x,y)=(x_{0}, y_{0})$. Since $\frac{\partial f}{\partial y}(x_{0}, y_{0},0) \neq  0$, the curve of singularities $S = \{f(x,y,0)=0\}$ of \eqref{eq-def-slowfast-1} for $\varepsilon=0$, near $(x,y)=(x_{0}, y_{0})$, can be represented as $y = \kappa (x)$ where $\kappa$ is a smooth function satisfying $\kappa(x_{0})=y_{0}$. Notice that $S$ has a quadratic contact with fast horizontal orbits of \eqref{eq-def-slowfast-1} with $\varepsilon=0$, at $(x,y)=(x_{0}, y_{0})$. Using \eqref{eq-def-hopf-turning-point}, it is clear that the singularities with $x$ close to $x_{0}$ and $x\ne x_{0}$ are normally hyperbolic, that is, 
$$\frac{\partial f}{\partial x}(x,\kappa(x),0) \ne 0,$$
for $x\ne x_{0}$. Then we can define the notion of slow vector field (see \cite[Chapter 3]{DDR-book-SF}) along normally hyperbolic portions of the curve of singularities near $x=x_{0}$ (its flow is often called the slow dynamics). More precisely, if we write $\tau=\varepsilon^l t$, then system \eqref{eq-def-slowfast-1} becomes
\begin{equation}\label{eq-def-slowfast-1-slow-time}
\left\{
\begin{array}{rcl}
\varepsilon^l {x}' & = & f(x,y,\varepsilon), \\
{y}' & = & g(x,y,\varepsilon),
\end{array}
\right.
\end{equation}
where the prime $'$ denotes the derivative of $x(\tau)$ and $y(\tau)$ with respect to the \textit{slow time} $\tau$. The systems \eqref{eq-def-slowfast-1} and \eqref{eq-def-slowfast-1-slow-time} are equivalent for $\varepsilon>0$. If we let $\varepsilon\to 0$ in \eqref{eq-def-slowfast-1-slow-time}, we obtain the slow system
\begin{equation}\label{SVF-final}
\left\{
\begin{array}{rcl}
0 & = & f(x,y,0), \\
{y}' & = & g(x,y,0).
\end{array}
\right.
\end{equation}

Using \eqref{SVF-final} and $y'=\kappa'(x)x'=-\frac{\frac{\partial f}{\partial x}(x,\kappa(x),0)}{\frac{\partial f}{\partial y}(x,\kappa(x),0)}x'$, we obtain the slow vector field
\begin{equation}
    \label{SD-final}
    {x}'  = -\frac{\frac{\partial f}{\partial y}(x,\kappa(x),0)}{\frac{\partial f}{\partial x}(x,\kappa(x),0)} g(x,\kappa(x),0).
\end{equation}

If $(x,y)=(x_{0}, y_{0})$ is a slow-fast Hopf point, then using \eqref{eq-def-hopf-turning-point} and L'Hospital's rule it follows that \eqref{SD-final} can be regularly extended through  $x=x_{0}$ and the slow dynamics points, locally near $x=x_{0}$, from the normally attracting branch to the normally repelling branch of $y=\kappa(x)$. This means that, for $\varepsilon>0$, orbits follow the attracting branch, pass through the slow-fast Hopf point, and then follow the repelling branch. See Figure \ref{fig-both-mechanisms}(a).

For classical Li\'enard Equations \eqref{classical}, we assume
\begin{equation}
    \label{Hopf-assumption}
    F(0)=F'(0)=0 \ \  \text{ and } \ \  \frac{F'(x)}{x}>0, \  \forall x\in [-\rho,\rho],
\end{equation}
where $\rho$ is a positive constant. 
The assumptions in \eqref{Hopf-assumption} imply that the curve of singularities $S$ contains
a normally attracting branch ($x\in (0,\rho]$), a normally repelling branch ($x\in[-\rho,0)$), and a slow-fast Hopf point at $(x, y) = (0, 0)$. In this case, the slow dynamics is given by 
\begin{equation}
    \label{slowdynamics}
    x'=\frac{dx}{d\tau}=-\frac{x}{F'(x)},
\end{equation}
where $\tau=\varepsilon^2 t$ is the slow time. Then    the slow dynamics points from the attracting part to the repelling part of $S$ near $x=0$. See Figure \ref{fig-Hopf}.

Denote the \textit{fast relation function} by $L_{H}$ (see \cite{Dbalanced}). For $x>0$, a canard cycle $\Gamma^H_x$ is the union of a fast orbit at height $y=F(x)$ and the part of the parabolic curve $S$ between the $\alpha$-limit point $(L_{H}(x),F(L_{H}(x)))$ and the $\omega$-limit point $(x,F(x))$ of the fast orbit. We have $L_{H}(x)<0$ and $F(x)=F(L_{H}(x))$. Such canard cycles are well-defined for $x\in (0,\min\{\rho, L_{H}^{-1}(-\rho)\}]$. See Figure \ref{fig-Hopf}.

\begin{figure}[ht]
	\begin{center}
		\includegraphics[width=6.6cm,height=5.8cm]{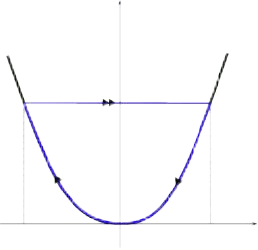}
		{\footnotesize \put(-25,109){$S$}
        \put(-89,88){$\Gamma^H_x$}
        \put(-35,10){$x$}
        \put(-178,8){$L_{H}(x)$}
        \put(2,16){$x$}
        \put(-97,165){$y$}
}
         \end{center}
	\caption{ A canard cycle $\Gamma^H_x$ created by the Hopf mechanism.}
	\label{fig-Hopf}
\end{figure}
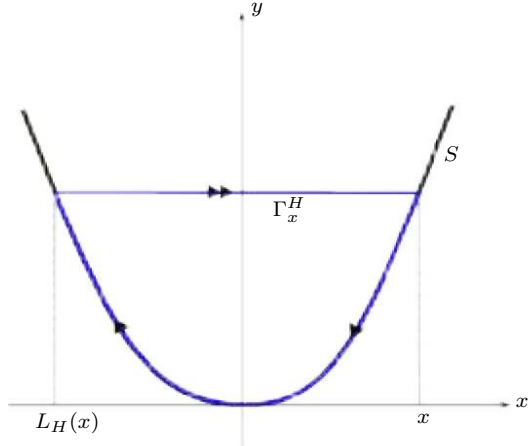

The slow divergence integral (see \cite[Chapter 5]{DDR-book-SF} and \cite{SDICLE1}) associated with $\Gamma^H_x$ is given by
 \begin{equation}
     \label{SDI-HopfM}
     I_{H}(x) := \int_{x}^{L_{H}(x)}\frac{(F'(s))^{2}}{s}ds, \ \quad \ x\in (0,\min\{\rho, L_{H}^{-1}(-\rho)\}].
 \end{equation}
This is the integral of the divergence of the vector field \eqref{classical} for $\varepsilon=0$ (which is equal to $-F'(x)$) with respect to the slow time $\tau$ (which is $d\tau=-\frac{F'(x)}{x}dx$).
 
A criterion for the existence of limit cycles produced by $\Gamma^H_x$ is given in Theorem \ref{thm-Hopf-simple-zeros}, and its proof can be found in \cite[Theorem 2]{SDICLE1}.

\begin{theorem}
    \label{thm-Hopf-simple-zeros}

    Suppose that $I_H(x)$ has exactly $k$ simple zeros $x_1<\cdots<x_k$ in $(0,\min\{\rho, L_{H}^{-1}(-\rho)\})$. Let $x_{k+1}\in (0,\min\{\rho, L_{H}^{-1}(-\rho)\}]$ satisfy $x_k<x_{k+1}$. Then there is a smooth function $a=a(\varepsilon)$ with $a(0) = 0$, so that the perturbed system

\begin{equation}\label{classicalHHopf}
                \begin{cases}
                    \dot{x} = y-F(x), \\
                    \dot{y} = \varepsilon^2\left( a(\varepsilon) -x\right),\nonumber
                \end{cases}
            \end{equation}
has $k + 1$ periodic orbits $\mathcal{O}^{x_i}_\varepsilon$, with $i=1,\dots,k+1$, for each $\varepsilon>0$ small enough. The periodic orbit $\mathcal{O}^{x_i}_\varepsilon$ is isolated, hyperbolic and Hausdorff close to the canard cycle $\Gamma^H_{x_i}$ as $\varepsilon\to 0$.
\end{theorem}

\subsection{The jump breaking mechanism}\label{subsection-jump}\noindent

We say that \eqref{eq-def-slowfast-1} has a (generic) \textit{jump point} (see \cite[Definition 2.3]{DDR-book-SF}) at $(x,y)=(x_{0}, y_{0})$ if 
\begin{equation}\label{eq-def-jump-point}
\begin{split}
    f(x_{0}, y_{0},0) = \frac{\partial f}{\partial x}(x_{0}, y_{0},0) = 0, \ \ \frac{\partial^{2} f}{\partial x^{2}}(x_{0}, y_{0},0) \neq 0,  \\ 
      \frac{\partial f}{\partial y}(x_{0}, y_{0},0) \ne  0 \  \text {  and  } \ g(x_{0}, y_{0},0)\neq 0. \qquad \ \ 
\end{split}
\end{equation}

If $(x,y)=(x_{0}, y_{0})$ is a jump point, we can define the slow dynamics \eqref{SD-final} in the same way as in Section \ref{subsection-Hopf}. However, in this case \eqref{eq-def-jump-point} implies that the vector field in \eqref{SD-final} is unbounded  near $(x,y)=(x_{0}, y_{0})$ and the orbits must jump. The slow dynamics is directed towards the jump point on both branches of $S$ or away from the jump point on both branches.

We say that \eqref{eq-def-slowfast-1} has a \textit{jump connection} if there are two jump points of \eqref{eq-def-slowfast-1} that are connected by a fast orbit and such that the curve of singularities $S$, locally near both jump points, is either concave up or concave down and the directions of the fast and slow dynamics are compatible (see Figure \ref{fig-connections}). The slow dynamics is therefore directed towards one jump point and away from the other one, and the function $g$ must have different sign near the jump points. Such a jump connection can be contained in canard cycles (see e.g. Figure \ref{fig-both-mechanisms}(b)). Of course, system \eqref{eq-def-slowfast-1} can have other types of jump connections, for instance, one jump point is concave up and the other one is concave down (see \cite[Section 6.2]{DDR-book-SF}). They are not considered in this paper.
 
 \begin{figure}[h!]
	\begin{center}
		\includegraphics[width=8.3cm,height=3.1cm]{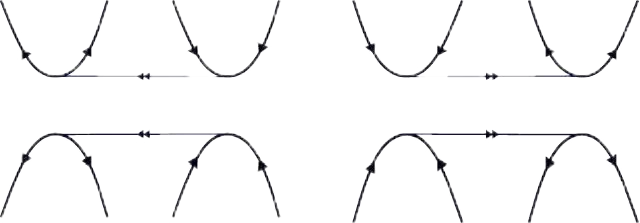}
         \end{center}
	\caption{Jump connections.}
	\label{fig-connections} 
\end{figure}



In the context of slow-fast Li\'enard equations \eqref{classical}, we assume that $F$ depends on a parameter $b\in\mathbb R$ kept close to $0$. We write
\begin{equation}\label{classicaljump}
    \begin{cases}
        \dot{x} = y-F_b(x), \\
        \dot{y} = -\varepsilon^2 x.
    \end{cases}
\end{equation}
            
A contact point $(x,F_b(x))\in S$ with $x\neq 0$ is a (generic) jump point if $F_b''(x)\neq 0$ (see \eqref{eq-def-jump-point}). For $b=0$, we assume that system \eqref{classicaljump} has a jump connection, that is, it has two jump points, which will be denoted by $p_-=(x_-,F_0(x_-))$ and $p_+=(x_+,F_0(x_+))$, and they are connected by a fast orbit $\gamma$. More precisely, the function $F_{0}$ has two minima of Morse type at $x = x_{-} < 0$ and $x = x_{+} > 0$ such that $F_{0}(x_{-}) = F_{0}(x_{+})$. We also suppose that 
\begin{equation}
    \label{jump-assumption}
     xF_0'(x)>0, \  \forall x\in [x_--\rho,x_-)\cup (x_+,x_++\rho],
\end{equation}
for a positive constant $\rho$.

The slow dynamics associated to \eqref{classicaljump} is also given by Equation \eqref{slowdynamics}. However, in this case, using the assumption \eqref{jump-assumption}, the slow dynamics goes towards $p_{+}$ and moves away from $p_{-}$. Moreover, $p_{-}$ (resp. $p_{+}$) is the $\omega$-limit (resp. the $\alpha$-limit) of the fast orbit $\gamma$. See Figure \ref{fig-jump}.

We now define a canard cycle $\Gamma^{J}_{x}$ for $x > x_{+}$ and $b = 0$. It is the union of a fast orbit at height $y = F_{0}(x)$, the fast orbit $\gamma$, the attracting portion of $S$ between $(x,F_{0}(x))\in S$ and $p_{+}\in S$, and the repelling portion of $S$ between $p_{-}\in S$ and $(L_{J}(x),F_{0}(L_{J}(x)))\in S$, where $L_{J}(x)$ is the fast relation function (see Figure \ref{fig-jump}). We have $L_{J}(x) < x_{-}$ and $F_{0}(x) = F_{0}(L_{J}(x))$. We assume that there is a constant  $\rho_{0} > x_{+}$ such that the canard cycle $\Gamma^{J}_{x}$ is well-defined for $x\in (\rho_{0},\min\{x_{+} + \rho, L_{J}^{-1}(x_{-} - \rho)\}]$, for some $\rho > 0$.

\begin{figure}[htb]
	\begin{center}
		\includegraphics[width=9cm,height=6cm]{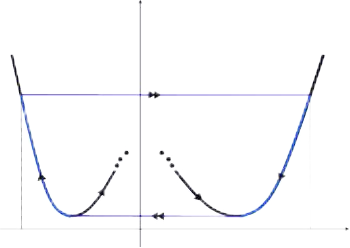}
		{\footnotesize \put(-20,118){$S$}
        \put(-129,95){$\Gamma^J_x$}
        \put(-31,6){$x$}
        \put(-250,3){$L_{J}(x)$}
        \put(2,14){$x$}
        \put(-149,165){$y$}
         \put(-80,15){$p_+$}
         \put(-208,15){$p_-$}
         \put(-159,24){$\gamma$}
}
         \end{center}
	\caption{ A canard cycle $\Gamma^J_x$ created by the jump mechanism, for $b=0$.}
	\label{fig-jump}
\end{figure}
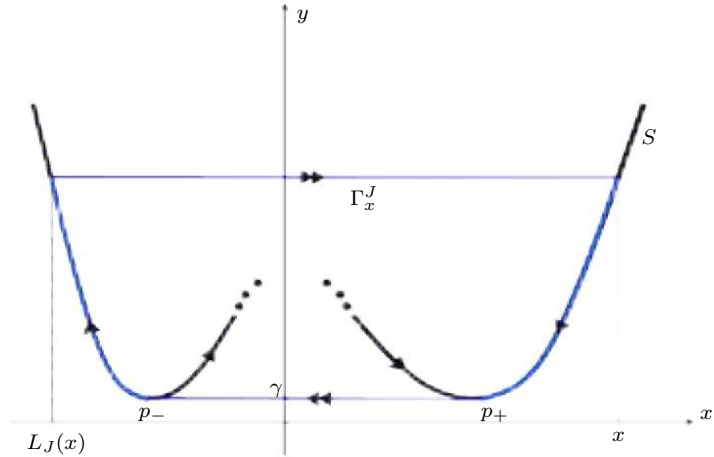


The jump points $p_{\pm}$ persist for all $b$ close to $0$, and we denote them by $p_{\pm}(b)$ with $p_{\pm}(0) = p_{\pm}$. We define
$$h(b) := F_{b}(x_{+}(b)) - F_{b}(x_{-}(b)).$$

Clearly, $h(0) = 0$. We assume that $b$ is a regular parameter for the jump breaking mechanism, that is, $h'(0)\neq 0$.

The slow divergence integral related to $\Gamma^{J}_{x}$ is given by
 \begin{equation}
     \label{SDI-JumpM}
     I_{J}(x) := \int_{x}^{x_{+}}\frac{(F_{0}'(s))^{2}}{s}ds + \int_{x_{-}}^{L_{J}(x)}\frac{(F_{0}'(s))^{2}}{s}ds,
 \end{equation}
 for $x\in (\rho_{0},\min\{x_{+} + \rho, L_{J}^{-1}(x_{-} - \rho)\}]$.

A proof of the following result (similar to Theorem \ref{thm-Hopf-simple-zeros}) can be found in \cite{Dbalanced}.
\begin{theorem}
    \label{thm-Jump-simple-zeros}
  Denote system \eqref{classicaljump} by $X_{\varepsilon,b}$ and  
    suppose that $I_J(x)$ defined by \eqref{SDI-JumpM} has exactly $k$ simple zeros. Then there is a smooth function $b=b(\varepsilon)$ with $b(0) = 0$, so that $X_{\varepsilon,b(\varepsilon)}$
has $k + 1$ periodic orbits for each $\varepsilon>0$ sufficiently small. All of them are isolated, hyperbolic and Hausdorff close to canard cycles $\Gamma^{J}_{x}$ as $\varepsilon\to 0$.
\end{theorem}

\subsection{Regularized PWL systems and generic breaking mechanisms}\label{sec-reg-mech}\noindent

In this section, we discuss how the breaking mechanisms studied in Sections \ref{subsection-Hopf} and \ref{subsection-jump} appear in regularizations of PWL vector fields (or more general piecewise smooth vector fields). The regularized system \eqref{LINREG-rescaled} is a special case of \eqref{eq-def-slowfast-1} with $l=1$ and  
 \begin{equation}
     \label{function-f}
     f(x,y,\varepsilon)=\displaystyle\frac{X_{1}(\varepsilon x,y) + Y_{1}(\varepsilon x,y)}{2} + \varphi(x)\frac{X_{1}(\varepsilon x,y) - Y_{1}(\varepsilon x,y)}{2},
 \end{equation}
and 
 \begin{equation}
     \label{function-g}
     g(x,y,\varepsilon)=\displaystyle\frac{X_{2}(\varepsilon x,y) + Y_{2}(\varepsilon x,y)}{2} + \varphi(x)\frac{X_{2}(\varepsilon x,y) - Y_{2}(\varepsilon x,y)}{2}.
 \end{equation}
 
Suppose that $(x,y)=(x_{0}, y_{0})$ is a slow-fast Hopf point or a jump point of \eqref{LINREG-rescaled}. Then Equation \eqref{function-f} implies that $\varphi'(x_{0}) = 0$ and $\varphi''(x_{0}) \neq 0$. Indeed, conditions $\frac{\partial f}{\partial x}(x_{0}, y_{0},0) = 0$ and $\frac{\partial^{2} f}{\partial x^{2}}(x_{0}, y_{0},0) \neq 0$ can be written as
$$\varphi'(x_{0})\left(X_1(0,y_{0})-Y_1(0,y_{0})\right)=0 \text{ and } \varphi''(x_{0})\left(X_1(0,y_{0})-Y_1(0,y_{0})\right)\ne 0. $$

This is equivalent to 
$$\varphi'(x_{0})=0, \ \varphi''(x_{0})\ne 0 \text{ and } X_1(0,y_{0})-Y_1(0,y_{0})\ne 0.$$

For the chosen $l = 1$ and functions $f$ and $g$, we get the following result.
\begin{proposition}\label{prop-properties}
    In Equation \eqref{eq-def-slowfast-1}, assume $l = 1$ and $f$, $g$ defined in \eqref{function-f} and \eqref{function-g}, respectively. Then the following statements are true.
    \begin{itemize}
        \item[(1)] For $x_{0}\in (-1,1)$, the point $(x_{0},y_{0})$ is not a slow-fast Hopf point. 
        \item[(2)] If $(x_{0},y_{0})$ and $(x_{1},y_{0})$ are generic jump points with $x_{0,1}\in (-1,1)$, then a jump connection is not possible.
    \end{itemize}
\end{proposition}
\begin{proof}
(1). Suppose that $(x_{0},y_{0})$ with $x_{0}\in (-1,1)$ is a slow-fast Hopf point, that is, the assumption \eqref{eq-def-hopf-turning-point} is satisfied. Conditions $\frac{\partial f}{\partial x}(x_{0}, y_{0},0) = 0$ and $\frac{\partial^{2} f}{\partial x^{2}}(x_{0}, y_{0},0) \neq 0$ imply $\varphi'(x_{0}) = 0$, and $\frac{\partial g}{\partial x}(x_{0}, y_{0},0)\ne 0$ and \eqref{function-g} imply $\varphi'(x_{0}) \neq 0$, which is a contradiction.

(2). Suppose that there is a jump connection between $(x_{0}, y_{0})$ and $(x_1,y_{0})$, that is, the assumption \eqref{eq-def-jump-point} is satisfied for both points and the orientation of the slow dynamics is compatible. Since $f(x_i,y_{0},0)=\frac{\partial f}{\partial x}(x_i,y_{0},0) = 0$ and $\frac{\partial^{2} f}{\partial x^{2}}(x_i,y_{0},0) \neq 0$, for $i=0,1$, it follows that  
$$X_1(0,y_{0})-Y_1(0,y_{0})\ne 0, \quad \text{ and } \quad \varphi(x_{0})=\varphi(x_1)=-\frac{X_1(0,y_{0})+Y_1(0,y_{0})}{X_1(0,y_{0})-Y_1(0,y_{0})}.$$

However, $\varphi(x_{0})=\varphi(x_1)$ and \eqref{function-g} imply $g(x_{0}, y_{0},0)=g(x_1,y_{0},0)$, therefore the orientation of the slow dynamics is not compatible, because the function $g$ has the same sign when evaluated in such points. Therefore we do not have a jump connection. 
\end{proof}

\begin{remark}
We cannot exclude the possibility of having slow-fast Hopf points or jump connections in \eqref{LINREG-rescaled}. Proposition \ref{prop-properties} only suggests that $l=1$ is not a good choice. As we shall see in the following, they can occur when $g(x,y,\varepsilon)=\varepsilon\tilde g(x,y,\varepsilon)$ for some smooth function $\tilde g$ (recall that $g $ is defined in \eqref{function-g}).
\end{remark}

Consider a PWL vector field of the form
\begin{equation}\label{eq-pwl-model}
Z(x,y) = \left\{
\begin{array}{rclc}
   X(x,y) & = & (a_{1} + b_{1} x + y, b_{2}x), & x > 0, \\ 
   Y(x,y) & = & (\alpha_{1} + \beta_{1} x + y, \beta_{2} x), & x < 0.
\end{array}
\right.
\end{equation}

A linear regularization \eqref{LINREG-rescaled} of $Z$ leads to the slow-fast Li\'enard equation
\begin{equation}\label{Lienard-generalized}
\left\{
\begin{array}{rcl}
   \dot{x} & = & y + \displaystyle\frac{1}{2}\Big( a_1+\alpha_1+(b_1+\beta_1)\varepsilon x + \big(a_1-\alpha_1+(b_1-\beta_1)\varepsilon x\big)\varphi(x)\Big), \\ 
   \dot{y} & = &\varepsilon^{2} \displaystyle\frac{x}{2}\big(b_2+\beta_2 + (b_2-\beta_2)\varphi(x)\big).
\end{array}
\right.
\end{equation}

The system \eqref{Lienard-generalized} is a special case of \eqref{eq-def-slowfast-1} with $l=2$. For $\varepsilon=0$, system \eqref{Lienard-generalized} has the curve of singularities $\{y=F(x):=-\frac{1}{2}\big( a_1+\alpha_1 +(a_1-\alpha_1)\varphi(x)\big)\}$. From \eqref{SD-final}, it follows that the slow dynamics is given by
\begin{equation}
    \label{SD-Lienard-generalized}
    x'=-\frac{x\big(b_2+\beta_2 + (b_2-\beta_2)\varphi(x)\big)}{(a_1-\alpha_1)\varphi'(x)}.\nonumber
\end{equation}

The main advantage of working with the system \eqref{Lienard-generalized} is that slow-fast Hopf points and jump connections can occur in \eqref{Lienard-generalized}. More precisely, we have the following result.
\begin{lemma}\label{lemma-classificatio}
 Consider system \eqref{Lienard-generalized}. The following statements hold.   
 \begin{itemize}
     \item[(1)] The point $(x_{0}, y_{0})=(0,F(0))$ is a slow-fast Hopf point if, and only if the equation \eqref{Lienard-generalized} satisfies 
     $$\varphi'(0) = 0, \quad \varphi''(0)\neq 0, \quad a_{1} \neq \alpha_{1} \quad  \text{and} \quad  b_{2} + \beta_{2} + (b_{2} - \beta_{2})\varphi(0) < 0.$$
     \item[(2)]  System \eqref{Lienard-generalized} must satisfy the following conditions in order to generate a jump connection, with jump points located at $(x_{0}, y_{0})$ and $(x_1,y_{0})$, $x_{0}<x_1$: There exists a regular orbit connecting $(x_{0}, y_{0})$ and $(x_1,y_{0})$ for $\varepsilon=0$, 
     \begin{align*}
         & \qquad a_1\ne \alpha_1, \ \  y_{0}=F(x_i), \ \ \varphi'(x_i)=0, \ \ \varphi''(x_i)\ne 0, \ \ i=0,1, \\
         &  x_{0}\big(b_2+\beta_2 + (b_2-\beta_2)\varphi(x_{0})\big)> 0, \ \ x_1\big(b_2+\beta_2 + (b_2-\beta_2)\varphi(x_1)\big)< 0, 
     \end{align*} 
 \end{itemize}
 and $\varphi''(x_0)$ and $\varphi''(x_1)$ have the same sign.
\end{lemma}
\begin{proof}
    This follows easily from \eqref{eq-def-hopf-turning-point}, \eqref{eq-def-jump-point}, \eqref{Lienard-generalized} and the definition of a jump connection. 
\end{proof}



\section{Proof of the main results}\label{section-proof-all}\noindent

In this section, we prove Theorems \ref{thm-main-Hopf} and \ref{thm-main-jump}. We point out that it is more convenient to deal with classical Li\'enard equations after regularization. Therefore, in the PWL system \eqref {eq-pwl-model} we set $b_{2} = \beta_{2} = -1$, and then we obtain $\dot y = -\varepsilon^{2} x$ in \eqref{Lienard-generalized}. We also assume that $b_{1} = \beta_{1} = 0$ such that the $x$-component of \eqref{Lienard-generalized} is independent of $\varepsilon$. Theorem \ref{thm-Hopf-simple-zeros} (resp. Theorem \ref{thm-Jump-simple-zeros}), stated for the classical Li\'enard equations, will play a crucial role in the proof of Theorem \ref{thm-main-Hopf} (resp. Theorem \ref{thm-main-jump}).

Finally, we remark that in this paper we do not focus on the description of the phase portraits for different values of the coefficients of $X$ and $Y$ in Equation \eqref{eq-pwl-model} because we do not use such a description in the proof of Theorems \ref{thm-main-Hopf} and \ref{thm-main-jump}. Recall that we are only interested in the limit cycles inside the regularization stripe. 

For the sake of readability, we rewrite system \eqref{LINREG-rescaled}:
\begin{equation}\label{LINREG-rescaled-rewrite}
\left\{
\begin{array}{rcl}
   \dot{x} & = & \displaystyle\frac{X_{1}(\varepsilon x,y) + Y_{1}(\varepsilon x,y)}{2} + \varphi(x)\frac{X_{1}(\varepsilon x,y) - Y_{1}(\varepsilon x,y)}{2}, \\ 
   \dot{y} & = & \varepsilon\left(\displaystyle\frac{X_{2}(\varepsilon x,y) + Y_{2}(\varepsilon x,y)}{2} + \varphi(x)\frac{X_{2}(\varepsilon x,y) - Y_{2}(\varepsilon x,y)}{2}\right).
\end{array}
\right.
\end{equation}

\subsection{Proof of Theorem \ref{thm-main-Hopf}}\label{section-proof-Hopf-main}\noindent

Define the function
\begin{equation}\label{F-Hopf-special}
F(x) := \frac{1}{2}x^{2} + \delta F_{o}(x),
\end{equation}
where $F_{o}$ is an odd function satisfying $F'_{o}(0) = 0$ and $\delta$ is a parameter kept close to zero. Notice that assumption \eqref{Hopf-assumption} is satisfied in any fixed segment $[-\rho,\rho]$ by taking $\delta$ small enough. By \cite[Proposition 1]{SDICLE1}, the slow divergence integral \eqref{SDI-HopfM} can be written as 
$$I_H(x) = -2\delta\left(F_{o}(x)  + O(\delta)\right).$$ 
Indeed, following the notation of \cite[Proposition 1]{SDICLE1}, in our case we have $F_{e}(x) = \frac{1}{2}x^{2}$ and $f_{e} = 1$. This implies that simple zeros of $F_o$ persist as simple zeros of $I_H$, for small $\delta\neq 0$. 

Define the PWL vector field
\begin{equation}\label{eq-pwl-proof-hopf}
Z(x,y,\alpha) = \left\{
\begin{array}{rclc}
   X(x,y,\alpha) & = & (-3 + y, \alpha - x), & x > 0, \\ 
   Y(x,y,\alpha) & = & (-1 + y, \alpha - x), & x < 0,
\end{array}
\right.
\end{equation}
where $\alpha$ is a breaking parameter kept close to zero. The points $(\alpha,3)$ and $(\alpha,1)$ are linear centers of $X$ and $Y$, respectively. Define $\psi(x) := F(x) - 2$, where $F$ is given in \eqref{F-Hopf-special}. For $\alpha =\varepsilon a$, with $a$ close to zero, consider the slow-fast system 
\begin{equation}\label{eq-Hopf-important}
\left\{
\begin{array}{rcl}
   \dot{x} & = & y - (\psi(x)+2), \\ 
   \dot{y} & = & \varepsilon^2( a-x).
\end{array}
\right.
\end{equation}

Recall that \eqref{eq-Hopf-important} is obtained from \eqref{LINREG-rescaled-rewrite}, with $\varphi=\psi$, but it is not yet a linear regularization of the PWL system \eqref{eq-pwl-proof-hopf} because $\psi$ is not a transition function. However, $\psi$ will be important for the construction of the transition function.

Let $k>0$ be an integer. Take 
$$F_{o}(x) = x^{3}(x^{2}-\tilde{x}_{1}^{2})\cdot \dots \cdot (x^{2} - \tilde{x}_{k}^{2}),$$
with $0 < \tilde{x}_{1} < \cdots <\tilde{x}_k$. We can assume that the simple zeros $\tilde{x}_1,\dots,\tilde{x}_{k}$ of $F_{o}$ are small enough and fixed so that the canard cycle $\Gamma^{H}_{\tilde{x}_i}$, associated with \eqref{eq-Hopf-important} for $a = 0$, is contained in the stripe $\{-1 < x < 1\}$ (or, equivalently, $\tilde{x}_{i} < 1$ and $L_{H}(\tilde{x}_{i}) > - 1$), for each $i=1,\dots,k$, and for each $\delta$ small. It is clear from the construction of $\psi$ and $F_{o}$ that
$$
\psi(0) = -2, \quad \psi'(0)=0, \quad \psi''(0) > 0, \quad \psi(-1) < -1, \quad \psi(1) < 1, 
$$
and $\psi'(x) > 0$ for $x\in (0,1]$ and $\psi'(x) < 0$ for $x\in [-1,0)$, for each $\delta$ sufficiently small. Now, we fix a $\delta\ne 0$ sufficiently small so that the above properties are satisfied and such that the slow divergence integral $I_H(x)$ has $k$ simple zeros $0<x_1\cdots<x_k$, with $x_i$ close to $\tilde{x}_i$, $ x_i<1$ and $L_{H}( x_i)>-1$, $i=1,\dots, k$.

We choose $x_{k+1} > x_{k}$ satisfying $x_{k+1}<1$ and $L_{H}(x_{k+1}) > -1$. From Theorem \ref{thm-Hopf-simple-zeros} it follows that there exists a smooth function $a = a_k(\varepsilon)$ with $a_k(0) = 0$, so that system \eqref{eq-Hopf-important}, with $a = a_k(\varepsilon)$, has $k + 1$ hyperbolic limit cycles $\mathcal{O}^{x_i}_\varepsilon$, with $i=1,\dots,k+1$, for each $\varepsilon>0$ small enough. Moreover, the limit cycle $\mathcal{O}^{x_i}_\varepsilon$ converges in the Hausdorff sense to the canard cycle $\Gamma^H_{x_i}$ as $\varepsilon\to 0$. This implies that there is a constant $\rho\in (0,1)$, with $\rho$ close to $1$, such that the limit cycle $\mathcal{O}^{x_i}_\varepsilon$ is contained in the stripe $\{-\rho<x<\rho\}$, for each $i=1,\dots,k+1$ and for each $\varepsilon>0$ sufficiently small. 
\begin{figure}[htb]
	\begin{center}
		\includegraphics[width=6cm,height=4cm]{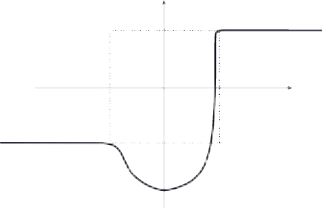}
		{\footnotesize 
        \put(-14,60){$x$}
        \put(-90,115){$y$}
}
         \end{center}
	\caption{ A transition function $\varphi_k$ with one critical point in $(-1,1)$, located at $x=0$.}
	\label{fig-transition-hopf}
\end{figure}

\subsubsection{Construction of the transition function}\noindent
\label{section-construct-trans}

Notice that the polynomial $\psi:\mathbb R\to \mathbb R$, defined above, is not a transition function. Our goal now is to construct a transition function $\varphi_{k}$ satisfying $\varphi_{k}(x) \equiv \psi(x)$, for all $x\in [-\rho,\rho]$, and such that $\varphi_{k}$ has no additional critical points, besides the one at $ x = 0$ (see Figure \ref{fig-transition-hopf}). Indeed, consider the \textit{cut-off functions} $A,B:\mathbb R\to\mathbb R$ (see \cite{Lovett}) satisfying
\begin{enumerate}
    \item Both $A,B$ are smooth;
    \item $A(x)=0$ for $x\le \rho$, $A(x)=1$ for $x\ge 1$, and $A'(x)>0$ for $x\in(\rho,1)$;
    \item $B(x)=0$ for $x\ge -\rho$, $B(x)=1$ for $x\le -1$, and $B'(x)<0$ for $x\in(-1,-\rho)$.
\end{enumerate}

Define  
$$\varphi_{k}(x) := A(x)\left(1-B(x)\right) - B(x) + \psi(x)\left(1-A(x)\right)\left(1-B(x)\right).$$

We claim that $\varphi_k(x)$ is a transition function. It can be checked that 
\begin{equation}
\varphi_k(x)=\left\{
\begin{array}{rcl}
   -1,&& \ x\le -1, \\ 
   -B(x)+\psi(x)\left(1-B(x)\right), && \ x\in(-1,-\rho),\\
   \psi(x), && \ x\in [-\rho,\rho], \\
   A(x)+\psi(x)\left(1-A(x)\right), && \ x\in(\rho,1), \\
   1, && \ x\ge 1.\nonumber
\end{array}
\right.
\end{equation}

It remains to prove that $\varphi_k'(x) < 0$, for $ x\in(-1,-\rho)$, and $\varphi_k'(x) > 0$, for $x\in(\rho,1)$. Let us first consider the interval $ x\in(-1,-\rho)$. Since $0 < B(x) < 1$, $B'(x)<0$, $\psi(x)<-1$ and $\psi'(x)<0$, for all $ x\in(-1,-\rho)$, we have
$$\varphi_k'(x)=-B'(x)\left(1+\psi(x)\right)+\psi'(x)\left(1-B(x)\right)<0,$$
for all $ x\in(-1,-\rho)$. Similarly, since $0<A(x)<1$, $A'(x)>0$, $\psi(x)<1$ and $\psi'(x)>0$, for all $ x\in(\rho,1)$, we have
$$\varphi_k'(x)=A'(x)\left(1-\psi(x)\right)+\psi'(x)\left(1-A(x)\right)>0,$$
for each $ x\in(\rho,1)$. Thus, $\varphi_k$ is a transition function with precisely one critical point in the interval $(-1,1)$, which is $x = 0$.

Consider the slow-fast system
 \begin{equation}\label{eq-Hopf-important-varph_k}
\left\{
\begin{array}{rcl}
   \dot{x} & = & y - (\varphi_k(x)+2), \\ 
   \dot{y} & = & \varepsilon^2( a-x).
\end{array}
\right.
\end{equation}

Notice that the vector fields in \eqref{eq-Hopf-important} and \eqref{eq-Hopf-important-varph_k} are equal in the stripe $\{-\rho<x<\rho\}$, and $\mathcal{O}^{x_i}_\varepsilon$ is therefore a (hyperbolic) limit cycle of \eqref{eq-Hopf-important-varph_k} with $a=a_k(\varepsilon)$, for each $i=1,\dots,k+1$ and for each $\varepsilon>0$ small enough. We conclude that the {$\varphi_k$-linear regularization \eqref{LINREG}, with $X,Y$ defined above, and $\alpha=\varepsilon a_k(\varepsilon)$, has at least $k+1$ hyperbolic limit cycles, for each $\varepsilon>0$ small enough. This completes the proof of Theorem \ref{thm-main-Hopf}.

\begin{remark}\label{remark-Cpm}
In the proof of Theorem \ref{thm-main-Hopf}, we used the PWL system \eqref{eq-pwl-proof-hopf}. Consider the following generalization of \eqref{eq-pwl-proof-hopf}:
\begin{equation}\label{eq-pwl-proof-hopf-2}
Z(x,y,\alpha) = \left\{
\begin{array}{rclc}
   X(x,y,\alpha) & = & (-c_{+} + y, \alpha - x), & x > 0, \\ 
   Y(x,y,\alpha) & = & (-c_{-} + y, \alpha - x), & x < 0,
\end{array}
\right.
\end{equation}
where $\alpha$ is a breaking parameter kept close to zero and $c_\pm\in \mathbb R$. The points $(\alpha,c_{+})$ and $(\alpha,c_{-})$ are linear centers of $X$ and $Y$, respectively. The PWL system \eqref{eq-pwl-proof-hopf} is a special case of \eqref{eq-pwl-proof-hopf-2} with $c_{+} = 3$ and $c_{-} = 1$. One can reproduce a completely analogous proof of Theorem \ref{thm-main-Hopf} using \eqref{eq-pwl-proof-hopf-2}, but for this we make the following assumption on $c_{\pm}$:
$$1 < 2c_{-} < 2c_{+}.$$

Define $\psi(x):= \frac{F(x) - C_{+}}{C_{-}}$, with $F(x)$ given by \eqref{F-Hopf-special} and  $C_{+}=\frac{c_{+} + c_{-} }{2}$ and $C_{-}=\frac{c_{+} - c_{-} }{2}$. This assumption implies that $0 < C_{-} < C_{+}$, and that $\psi(-1) < -1$ and $\psi(1) < 1$, for $\delta$ sufficiently small. Moreover, using the definition of $F$, we have $\psi(0) < -1$, $\psi'(0) = 0$, $\psi''(0) > 0$, $\psi'(x) > 0$ for $x\in (0,1]$ and $\psi'(x) < 0$ for $x\in [-1,0)$, for each $\delta$ sufficiently small.

Now, with the above assumption in mind, one can prove Theorem \ref{thm-main-Hopf} in a similar way using \eqref{eq-pwl-proof-hopf-2}. However, the proof is given for $c_{+} = 3$ and $c_{-} = 1$ for the sake of  readability. 


\end{remark}

\subsection{Proof of Theorem \ref{thm-main-jump}}\label{section-proof-jump-main}\noindent


Consider the function
\begin{equation}
    \label{Fb-special}
    F_b(x):=P_{e}(x) + bx + \delta P_{o}(x),
\end{equation}
in which $b$ is a breaking parameter kept close to $0$, $\delta$ is a perturbation parameter close to $0$, $P_{e}$ is given by 
$$P_{e}(x) = \displaystyle\frac{x^{4}}{4} - \frac{\eta^{2}}{2}x^{2},$$
where $\eta$ is a positive and fixed constant satisfying $0 < \eta < \frac{\sqrt{2}}{2}$, and $P_{o}$ is $C^\infty$-smooth in $\mathbb R$ and odd in the symmetric set $D:=(-\infty,-\eta]\cup [\eta,\infty)$ (that is, $P_{o}(-x)=-P_{o}(x)$ for all $x\in D$). Furthermore, we assume that $P_{o}(\eta)=P'_{o}(\eta)=0$. Since $P_{o}$ is odd in $D$, we have $P_{o}(-\eta)=P'_{o}(-\eta)=0$.

 The even polynomial $P_{e}$ has a zero of multiplicity two positioned at the origin, and two simple zeros $z_{\pm} = \pm\sqrt{2}\eta\in (-1, 1)$. It is important to note that, for $b = 0$, the points $x_{\pm} = \pm\eta$ are critical points of Morse type (both minima) of $F_{0}$ and $F_{0}(x_{-}) = F_{0}(x_{+})$, for every $\delta$ small enough ($x_{\pm}$ do not depend on $\delta$). We refer to Figure \ref{fig-graph-pe}. 

 Using the properties of $P_{e}$ and $P_{o}$, we can directly see that $F_0'(x)>0$ for $x\in (\eta,1]$ and $F_0'(x)<0$ for $x\in [-1,-\eta)$, for every $\delta$ sufficiently small (see assumption \eqref{jump-assumption}).

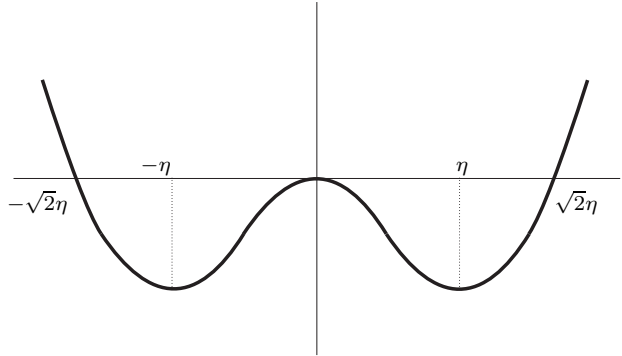
\begin{figure}[htb]
	\begin{center}
		\includegraphics[width=8cm]{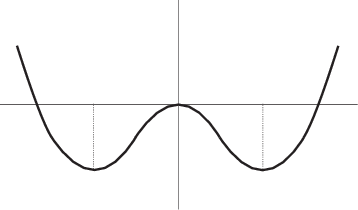}
       {\footnotesize \put(-62,70){$\eta$}
        \put(-180,70){$-\eta$}
        \put(-25,55){$\sqrt{2}\eta$}
        \put(-230,55){$-\sqrt{2}\eta$}}
         \end{center}
	\caption{Graph of the function $F_{b}$ defined in \eqref{Fb-special} for $b =\delta= 0$. The points $x_{\pm} = \pm\eta$ are critical points and $z_{\pm} = \pm\sqrt{2}\eta$ are simple zeros. All these points are inside the open interval $(-1,1)$.}
	\label{fig-graph-pe}
\end{figure}


Consider a PWL vector field of the form
\begin{equation}\label{eq-jump-pwl}
Z(x,y) = \left\{
\begin{array}{rcl}
   X(x,y) & = & (y - 3, -x), \\ 
   Y(x,y) & = & (y - 1, -x).
\end{array}
\right.
\end{equation}

The singular points $(0,3), (0,1)\in\Sigma$ are linear centers of $X,Y$, respectively. In Equation \eqref{LINREG-rescaled-rewrite}, we set the function $\varphi(x)=\psi_{b}(x)$ satisfying 
\begin{equation}\label{eq-jump-psi}
\psi_{b}(x) := F_{b}(x) -2.    
\end{equation}

Then, using \eqref{eq-jump-pwl}, we obtain the slow-fast system
\begin{equation}\label{eq-pwl-centers}
\left\{
\begin{array}{rcl}
   \dot{x} & = & y - F_{b}(x), \\ 
   \dot{y} & = & -\varepsilon^{2} x.
\end{array}
\right.
\end{equation}

System \eqref{eq-pwl-centers} is not yet a regularization for \eqref{eq-jump-pwl}, because $\psi_{b}$ in \eqref{eq-jump-psi} is not a transition function. Straightforward calculations lead to $F_{b}'(x) = \psi_{b}'(x)$, where the prime $'$ denotes the derivative with respect to the variable $x$. 

The slow divergence integral \eqref{SDI-JumpM} has the form
\begin{equation}\label{SDI-JumpM-reg}
I_{J}(x)
= \int_{x}^{\eta} \frac{\bigl(F_{0}'(s)\bigr)^2}{s}  ds
+ \int_{-\eta}^{L_{J}(x)} \frac{\bigl(F_{0}'(s)\bigr)^2}{s}  ds.
\end{equation}

We keep $x$ in \eqref{SDI-JumpM-reg} in a compact interval $D_1\subset (\sqrt{2}\eta,1)$. Then the fast relation function $L_{J}(x)$ and the canard cycle $\Gamma^J_x$ are well defined, and $L_{J}(x)\in (-1,-\sqrt{2}\eta)$, for all $x\in D_1$ and every $\delta$ sufficiently small (see Figures \ref{fig-jump} and \ref{fig-graph-pe}).

\subsubsection{Regularity of the breaking parameter}\noindent

We must verify that $b$ is a regular breaking parameter, and it is enough to prove this for $\delta = 0$. We know that the jump points $x_\pm$ persist for all $b$ close to $0$. Using asymptotic expansions in $b$ and the notation introduced in Section \ref{subsection-jump}, we can write the solutions $x_{\pm}(b)$ of $F_{b}'(x) = 0$ as
$$
x_{+}(b) = \eta - \displaystyle\frac{b}{2\eta^{2}} - \frac{3b^{2}}{8\eta^{5}} + O(b^{3}), \quad \quad x_{-}(b) = - \eta - \displaystyle\frac{b}{2\eta^{2}} + \frac{3b^{2}}{8\eta^{5}} + O(b^{3}).
$$

It easily follows that
$$h(b) = F_{b}\left(x_{+}(b)\right) - F_{b}\left(x_{-}(b)\right) = 2\eta b + O(b^{2}).$$

Therefore $h'(0)= 2\eta \neq 0$ and it is a submersion. In other words, $b$ is a regular breaking parameter. 

\subsubsection{Asymptotics of the fast relation function}\noindent

The next step is to find the asymptotics of the fast relation function $L_{J}(x)$. Recall that $b=0$. Assuming $ \delta = 0$, it follows that $F_{0}$ in \eqref{Fb-special} is an even function. Therefore, we can write $L_{J}(x) = -x + \delta L_{1}(x) + O(\delta^{2})$. In addition, since $P_{o}$ is odd in $D$ and $P_{e}$ is even, it is true that, for $x\in D_1\subset D$,
$$
P_{o}(L_{J}(x)) = -P_{o}(x) + O(\delta), \quad \quad 
P_{e}(L_{J}(x)) = P_{e}(x) - \delta P_{e}'(x)L_{1}(x) + O(\delta^{2}).    
$$

Therefore, from $F_{0}(x) = F_{0}(L_{J}(x))$ one obtains
\begin{equation*}
P_{e}(x) + \delta P_{o}(x) = P_{e}(x) - \delta P_{e}'(x)L_{1}(x)- \delta P_{o}(x) + O(\delta^{2}),
\end{equation*}
and finally one obtains
\begin{equation}\label{L1-expression}
    L_{1}(x) = \displaystyle\frac{-2P_{o}(x)}{P_{e}'(x)}=\frac{-2P_{o}(x)}{x(x^2-\eta^2)}, \ x\in D_1.
    \end{equation}


\subsubsection{Asymptotics of the slow divergence integral}\noindent

Now, we will study asymptotics of the slow divergence integral $I_J(x)$ given by \eqref{SDI-JumpM-reg}. Recall that $x\in D_1$ and $\delta$ is close to zero. 

We have
$$
F_{0}'(s) = P_{e}'(s) + \delta P_{o}'(s), \quad \quad 
\left(F_{0}'(s)\right)^{2} = \left( P_{e}'(s) \right)^{2} + 2\delta P_{e}'(s)P_{o}'(s) + O(\delta^{2}).
$$

Using the expression for $P_{e}$, we can write
\begin{equation}\label{eq-jump-sdi-asymp}
\begin{array}{rcl}
   \displaystyle\int_{x}^{\eta} \frac{(F_{0}'(s))^{2}}{s} ds  & = & \displaystyle\int_{x}^{\eta} s(s^{2} - \eta^{2})^{2}ds
+ 2 \delta \int_{x}^{\eta}(s^{2} - \eta^{2}) P_{o}'(s) ds + O(\delta^{2}), \\
   & & \\
    \displaystyle\int_{-\eta}^{L_{J}(x)} \frac{(F_{0}'(s))^{2}}{s} ds & = & \displaystyle\int_{-\eta}^{L_{J}(x)} s(s^{2} - \eta^{2})^{2}ds
+ 2 \delta \int_{-\eta}^{L_{J}(x)}(s^{2} - \eta^{2})P_{o}'(s) ds\\
& & \\
    & & + O(\delta^{2}).
\end{array}
\end{equation}

Now, we will handle the integrals of Equation \eqref{eq-jump-sdi-asymp}. Firstly, using straightforward computations and \eqref{L1-expression}, one can show that 
\begin{equation}\label{eq-jump-sdi-asymp-1}
\begin{array}{rcl}
\displaystyle\int_{x}^{\eta} s(s^{2} - \eta^{2})^{2}ds + \displaystyle\int_{-\eta}^{L_{J}(x)} s(s^{2} - \eta^{2})^{2}ds & = & -\delta x(x^{2} - \eta^{2})^{2}L_{1}(x) + O(\delta^{2}) \\
& = &  2\delta(x^{2} - \eta^{2}) P_{o}(x) + O(\delta^{2}).
\end{array}
\end{equation}

Recall that the function $ P_{o}$ is odd in $D$. This, $L_{J}(x)=-x+O(\delta)$ and partial integration imply
\begin{equation}\label{eq-jump-sdi-asymp-2}
   \begin{array}{rcl}
       \displaystyle\int_{x}^{\eta} (s^{2} - \eta^{2}) P_{o}'(s)ds & = & - (x^{2} - \eta^{2})P_{o}(x)
- \displaystyle\int_{x}^{\eta} 2sP_{o}(s) ds, \\
& & \\
      \displaystyle\int_{-\eta}^{L_{J}(x)} (s^{2} - \eta^{2}) P_{o}'(s)ds  & = & - (x^{2} - \eta^{2})P_{o}(x)
- \displaystyle\int_{-\eta}^{L_{J}(x)} 2s P_{o}(s) ds + O(\delta).
   \end{array} 
\end{equation}

We also have
$$
- \displaystyle\int_{-\eta}^{L_{J}(x)} 2s P_{o}(s) ds = - \displaystyle\int_{x}^{\eta} 2s P_{o}(s) ds + O(\delta).
$$

Therefore, combining Equations \eqref{eq-jump-sdi-asymp}, \eqref{eq-jump-sdi-asymp-1} and \eqref{eq-jump-sdi-asymp-2}, one obtains the following asymptotics for the slow divergence integral $I_{J}$:
\begin{equation*}
\begin{array}{rcl}
I_{J}(x)
& = & \displaystyle\int_{x}^{\eta} \frac{\bigl(F_{0}'(s)\bigr)^2}{s}  ds
+ \int_{-\eta}^{L_{J}(x)} \frac{\bigl(F_{0}'(s)\bigr)^2}{s}  ds \\
& & \\
& = & 2\delta(x^{2} - \eta^{2})P_{o}(x) - 4 \delta\left((x^{2} - \eta^{2})P_{o}(x) + \displaystyle\int_{x}^{\eta} 2s P_{o}(s) ds\right) + O(\delta^{2}) \\
& & \\
& = & -2\delta\left((x^{2} - \eta^{2})P_{o}(x) + 2 \displaystyle\int_{x}^{\eta} 2s P_{o}(s) ds\right) + O(\delta^{2}) \\

& & \\

& = & 2\delta \displaystyle \int_{\eta}^{x} \Big{(} 2sP_{o}(s) - (s^{2}-\eta^{2})P_{o}'(s) \Big{)} ds + O(\delta^{2}),
\end{array}
\end{equation*}
with $x\in D_1$. In the last step, we used partial integration.

Therefore, simple zeros of
\begin{equation}\label{eq-int-roots}
I_{1}(x) = \displaystyle \int_{\eta}^{x} \Big{(} 2sP_{o}(s) - (s^{2}-\eta^{2})P_{o}'(s) \Big{)} ds
\end{equation}
 persist as simple zeros of $I_{J}$, for small $\delta\ne 0$.

\subsubsection{Constructing zeros of the slow divergence integral}\noindent

First, we show that there is a smooth function $P_{o}:\mathbb R\to\mathbb R$ ($P_{o}$ is odd in $D$ and $P_{o}(\eta)=P'_{o}(\eta)=0$) such that the associated integral function $I_{1}$, defined by \eqref{eq-int-roots}, has $k$ simple zeros in the interior of $D_1$. Recall that $D_1\subset (\sqrt{2}\eta,1)$.

Consider a polynomial
$$\widetilde P(x) = (x^2-\eta^2)^3(x-\tilde{x}_{1})\cdot \dots \cdot (x - \tilde{x}_{k}),$$
with $\tilde{x}_{1} < \cdots <\tilde{x}_k$ contained in the interior of $D_1$. For $x\ge \eta $, $P_{o}$ is defined as follows:
\begin{equation}\label{eq-solution}
    P_{o}(x) := - (x^{2} - \eta^{2})\displaystyle\int_{\eta}^{x}\displaystyle\frac{\widetilde P'(s)}{(s^{2} - \eta^{2})^{2}}ds.
\end{equation}
From \eqref{eq-solution} it follows that $P_{o}(\eta)=P'_{o}(\eta)=0$.
Since $P_{o}$ has to be odd in $D$, we define $P_{o}(x):=-P_{o}(-x)$, for $x\le -\eta$. Using cut-off functions} (see \cite{Lovett}), it is not difficult to see that $P_{o}$ can be smoothly extended to $\mathbb R$ (notice that the behavior of $P_{o}$ for $x\in (-\eta,\eta)$ is not relevant when we study zeros of $I_J$). We fix this $P_{o}$.

Now, it suffices to prove that $I_1(x)=\widetilde P(x)$, for $x\ge \eta$. This will imply that $I_1$ has $k$ simple zeros $\tilde{x}_{1}, \dots,\tilde{x}_k$. Indeed, we have 
$$2sP_{o}(s) - (s^{2}-\eta^{2})P_{o}'(s)=\widetilde P'(s), \ s\ge\eta.$$
It follows from \eqref{eq-solution}. Using \eqref{eq-int-roots} and $\widetilde P(\eta)=0$, we conclude that $I_1(x)=\widetilde P(x)$, for $x\ge \eta$.

\smallskip

We know that the simple zeros of $I_1$ persist as simple zeros of $I_{J}$, for $\delta \neq 0$ sufficiently small. More precisely, we can fix a small $\delta\ne 0$ so that the slow divergence integral $I_J(x)$ has $k$ simple zeros $x_1<\cdots<x_k$ contained in the interior of $D_1$, with $x_i$ close to $\tilde{x}_i$, $i=1,\dots, k$.


Now, as well as in the Hopf case, take $x_{k+1} > x_{k}$ satisfying $x_{k+1}\in D_1$. It follows from Theorem \ref{thm-Jump-simple-zeros} that there exists a smooth function $b = b_{k}(\varepsilon)$ with $b_{k}(0) = 0$ such that system \eqref{eq-pwl-centers}, with $F_{b} = F_{b_{k}(\varepsilon)}$, has $k + 1$ hyperbolic limit cycles $\mathcal{O}^{x_i}_\varepsilon$, with $i=1,\dots,k+1$, for each $\varepsilon>0$ small enough. Moreover, the limit cycle $\mathcal{O}^{x_i}_\varepsilon$ is Hausdorff close to the canard cycle $\Gamma^J_{x_i}$ as $\varepsilon\to 0$. Notice that the $k+1$ limit cycles are contained in the stripe $\{-1<x<1\}$. 

\subsubsection{Constructing the transition function}\noindent

The construction of a suitable transition function $\varphi_{b,k}$ from $\psi_{b}$, defined in \eqref{eq-jump-psi}, can be done in the same fashion as in Section \ref{section-proof-Hopf-main}. Observe that $\psi_{b}(-1) < -1$ and $\psi_{b}(1) < 1$, for $b$ and $\delta$ close to zero. We also have $\psi_{0}'(x) < 0$ in $[-1,-\rho]$ and $\psi_{0}'(x) > 0$ in $[\rho, 1]$, where $\rho\in (0,1)$, with $\rho$ close to $1$ and $x_{k+1}, L_J(x_{k+1})\in (-\rho,\rho)$, and $\delta$ is close to zero. It follows from the property of $F_0'(x)$ given before \eqref{eq-jump-pwl}. Notice that $\psi_{b}$ has precisely three critical points in $(-\rho,\rho)$ and they are of Morse type, for $b$ and $\delta$ close to zero.

\begin{figure}[htb]
	\begin{center}
		\includegraphics[width=8cm]{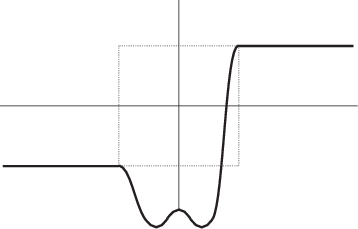}
         \end{center}
	\caption{Transition function $\varphi_{b,k}$ with three critical points in $(-1,1)$.}
	\label{fig-transition-jump}
\end{figure}

Now, following the steps of Section \ref{section-proof-Hopf-main}, we replace $\psi_{b}$ with a non-monotonic transition function $\varphi_{b,k}$ having precisely $3$ critical points, see Figure \ref{fig-transition-jump}. The $\varphi_{b,k}$-linear regularization of the PWL system \eqref{eq-jump-pwl} has the form
\begin{equation}
\left\{
\begin{array}{rcl}
   \dot{x} & = & y - (\varphi_{b(\varepsilon),k}(x)+2), \\ 
   \dot{y} & = & -\varepsilon^{2} x,
\end{array}
\right.
\end{equation}
and this system has $k+1$ hyperbolic limit cycles. This completes the proof of Theorem \ref{thm-main-jump}.

\begin{remark}
Consider
\begin{equation}\label{eq-pwl-proof-jump-2}
Z(x,y) = \left\{
\begin{array}{rclc}
   X(x,y) & = & (-c_{+} + y, - x), & x > 0, \\ 
   Y(x,y) & = & (-c_{-} + y, - x), & x < 0,
\end{array}
\right.
\end{equation}

The points $(0,c_{+})$ and $(0,c_{-})$ are linear centers of $X$ and $Y$, respectively (see Remark \ref{remark-Cpm}). The PWL system \eqref{eq-jump-pwl} is a special case of \eqref{eq-pwl-proof-jump-2} with $c_{+} = 3$ and $c_{-} = 1$. One can reproduce a completely analogous proof of Theorem \ref{thm-main-jump} using \eqref{eq-pwl-proof-jump-2}, under the following assumption on $c_{\pm}$:
$$\frac{1}{4} - \frac{\eta^{2}}{2} < 
    c_{-} < c_{+}.$$

If we define $\psi_{b}(x):= \frac{F_{b}(x) - C_{+}}{C_{-}}$, with $F_{b}(x)$ given in \eqref{Fb-special} and $C_\pm$ introduced in Remark \ref{remark-Cpm}, then it is not difficult to see that such a $\psi_{b}$ has similar properties as the function $\psi_{b}$ defined in \eqref{eq-jump-psi}. 
    


\end{remark}

\section*{Declarations}
 
\textbf{Ethical Approval:} Not applicable.
 \\
\\
 \textbf{Competing interests:} The authors declare that they have no conflict of interest.\\
 \\
\textbf{Funding:} The research of R. Huzak was supported by: Croatian Science Foundation (HRZZ) grant IP-2022-10-9820. Otavio Henrique Perez is supported by Sao Paulo Research Foundation (FAPESP) grant 2021/10198-9. \\
 \\
\textbf{Availability of data and materials:} Not applicable.

\bibliographystyle{plain}
\bibliography{bibtex}

@preamble{
   "\def\cprime{$'$} "
}

@article {Dbalanced,
    AUTHOR = {Dumortier, F.},
     TITLE = {Slow divergence integral and balanced canard solutions},
   JOURNAL = {Qual. Theory Dyn. Syst.},
  FJOURNAL = {Qualitative Theory of Dynamical Systems},
    VOLUME = {10},
      YEAR = {2011},
    NUMBER = {1},
     PAGES = {65--85},
      ISSN = {1575-5460},
   MRCLASS = {34C26 (34C23 34E13 34E17)},
  MRNUMBER = {2773292 (2012c:34125)},
MRREVIEWER = {Oleg Gendelman},
       DOI = {10.1007/s12346-011-0038-9},
       URL = {http://dx.doi.org/10.1007/s12346-011-0038-9},
}

@article {DPR,
    AUTHOR = {Dumortier, F. and Panazzolo, D. and Roussarie, R.},
     TITLE = {More limit cycles than expected in {L}i\'enard equations},
   JOURNAL = {Proc. Amer. Math. Soc.},
  FJOURNAL = {Proceedings of the American Mathematical Society},
    VOLUME = {135},
      YEAR = {2007},
    NUMBER = {6},
     PAGES = {1895--1904 (electronic)},
      ISSN = {0002-9939},
     CODEN = {PAMYAR},
   MRCLASS = {34C07 (34C05 34C26 37C10)},
  MRNUMBER = {MR2286102 (2007m:34081)},
MRREVIEWER = {Rafael Prohens},
       DOI = {10.1090/S0002-9939-07-08688-1},
       URL = {http://dx.doi.org/10.1090/S0002-9939-07-08688-1},
}

@incollection {smale,
    AUTHOR = {Smale, S. },
     TITLE = {Mathematical problems for the next century},
 BOOKTITLE = {Mathematics: frontiers and perspectives},
     PAGES = {271--294},
 PUBLISHER = {Amer. Math. Soc.},
   ADDRESS = {Providence, RI},
      YEAR = {2000},
   MRCLASS = {00A05 (00-02 01A61 01A67)},
  MRNUMBER = {MR1754783 (2001i:00003)},
}

@article {DDMoreLC,
    AUTHOR = {De Maesschalck, P. and Dumortier, F.},
     TITLE = {Classical  {L}i\'{e}nard equations of degree $n\ge 6$ can have $[\frac{n-1}{2}]+2$ limit cycles},
   JOURNAL = {J. Differential Equations},
  FJOURNAL = {Journal of Differential Equations},
    VOLUME = {250},
      YEAR = {2011},
    NUMBER = {4},
     PAGES = {2162--2176},
}

@article {SDICLE1,
    AUTHOR = {De Maesschalck, P. and Huzak, R.},
     TITLE = {Slow divergence integrals in classical {L}i\'{e}nard equations
              near centers},
   JOURNAL = {J. Dynam. Differential Equations},
  FJOURNAL = {Journal of Dynamics and Differential Equations},
    VOLUME = {27},
      YEAR = {2015},
    NUMBER = {1},
     PAGES = {177--185},
      ISSN = {1040-7294},
   MRCLASS = {34C07 (34E15)},
  MRNUMBER = {3317395},
MRREVIEWER = {Jes\'{u}s S. P\'{e}rez del R\'{\i}o},
       DOI = {10.1007/s10884-014-9358-1},
       URL = {https://doi.org/10.1007/s10884-014-9358-1},
}

@article {Program,
    AUTHOR = {Dumortier, F. and Roussarie, R. and Rousseau, C.},
     TITLE = {Hilbert's 16th problem for quadratic vector fields},
   JOURNAL = {J. Differential Equations},
  FJOURNAL = {Journal of Differential Equations},
    VOLUME = {110},
      YEAR = {1994},
    NUMBER = {1},
     PAGES = {86--133},
      ISSN = {0022-0396},
     CODEN = {JDEQAK},
   MRCLASS = {58F21 (34C05)},
  MRNUMBER = {1275749},
MRREVIEWER = {Carmen Chicone},
       DOI = {10.1006/jdeq.1994.1061},
       URL = {http://dx.doi.org/10.1006/jdeq.1994.1061},
}

@book {DDR-book-SF,
    AUTHOR = {De Maesschalck, P. and Dumortier, F. and Roussarie,
              R.},
     TITLE = {Canard cycles---from birth to transition},
    SERIES = {Ergebnisse der Mathematik und ihrer Grenzgebiete. 3. Folge. A
              Series of Modern Surveys in Mathematics [Results in
              Mathematics and Related Areas. 3rd Series. A Series of Modern
              Surveys in Mathematics]},
    VOLUME = {73},
 PUBLISHER = {Springer, Cham},
      YEAR = {2021},
     PAGES = {xxi+408},
      ISBN = {978-3-030-79232-9; 978-3-030-79233-6},
   MRCLASS = {34-02 (34C07 34D15 34E15 34E17)},
  MRNUMBER = {4304039},
}

@Article{DHGener,
 Author = {Huzak, R. and De Maesschalck, P.},
 Title = {Slow divergence integrals in generalized {Li{\'e}nard} equations near centers},
 FJournal = {Electronic Journal of Qualitative Theory of Differential Equations},
 Journal = {Electron. J. Qual. Theory Differ. Equ.},
 ISSN = {1417-3875},
 Volume = {2014},
 Pages = {10},
 Note = {Id/No 66},
 Year = {2014},
 Language = {English},
 DOI = {10.14232/ejqtde.2014.1.66},
 URL = {www.math.u-szeged.hu/ejqtde/p3307.pdf},
 zbMATH = {6439110},
 Zbl = {1324.34075}
}

@Article{RHKK,
 Author = {Huzak, R. and Kristiansen, K. Uldall},
 Title = {The number of limit cycles for regularized piecewise polynomial systems is unbounded},
 FJournal = {Journal of Differential Equations},
 Journal = {J. Differ. Equations},
 ISSN = {0022-0396},
 Volume = {342},
 Pages = {34--62},
 Year = {2023},
 Language = {English},
 DOI = {10.1016/j.jde.2022.09.028},
 zbMATH = {7615163}
}

@article{esteban2021a,
  author = {Esteban, M. and Llibre, J. and Valls, C.},
  title = "{The 16th Hilbert problem for discontinuous piecewise isochronous centers of degree one or two separated by a straight line}",
  language = {eng},
  format = {article},
  journal = {Chaos},
  volume = {31},
  number = {4},
  pages = {043112},
  year = {2021},
  issn = {10897682, 10541500},
  publisher = {American Institute of Physics Inc.},
  doi = {10.1063/5.0023055}
}

@article{li2021a,
  author = {Li, T. and Llibre, J.},
  title = "{On the 16th Hilbert Problem for Discontinuous Piecewise Polynomial Hamiltonian Systems}",
  language = {eng},
  format = {article},
  journal = {Journal of Dynamics and Differential Equations},
  pages = {1-16},
  year = {2021},
  issn = {15729222, 10407294},
  publisher = {Springer},
  doi = {10.1007/s10884-021-09967-3}
}

@article{llibre2013a,
  author = {Llibre, J. and Teixeira, M. A. and Torregrosa, J.},
  title = {Lower bounds for the maximum number of limit cycles of discontinuous piecewise linear differential systems with a straight line of separation},
  language = {eng},
  format = {article},
  journal = {International Journal of Bifurcation and Chaos},
  volume = {23},
  number = {4},
  pages = {1350066},
  year = {2013},
  issn = {17936551, 02181274},
  publisher = {World Scientific Publishing Co. Pte Ltd},
  doi = {10.1142/S0218127413500661}
}

@Article{Carmona,
 Author = {Carmona, V. and Fern{\'a}ndez-S{\'a}nchez, F.},
 Title = {Integral characterization for {Poincar{\'e}} half-maps in planar linear systems},
 FJournal = {Journal of Differential Equations},
 Journal = {J. Differ. Equations},
 ISSN = {0022-0396},
 Volume = {305},
 Pages = {319--346},
 Year = {2021},
 Language = {English},
 DOI = {10.1016/j.jde.2021.10.010},
 zbMATH = {7423292}
}

@Article{Freire,
 Author = {Freire, E. and Ponce, E. and Rodrigo, F. and Torres, F.},
 Title = {Bifurcation sets of continuous piecewise linear systems with two zones},
 FJournal = {International Journal of Bifurcation and Chaos in Applied Sciences and Engineering},
 Journal = {Int. J. Bifurcation Chaos Appl. Sci. Eng.},
 ISSN = {0218-1274},
 Volume = {8},
 Number = {11},
 Pages = {2073--2097},
 Year = {1998},
 Language = {English},
 DOI = {10.1142/S0218127498001728},
 zbMATH = {1408619},
 Zbl = {0996.37065}
}

@Article{LlibreOrd,
 Author = {Llibre, J. and Ord{\'o}{\~n}ez, M. and Ponce, E.},
 Title = {On the existence and uniqueness of limit cycles in planar continuous piecewise linear systems without symmetry},
 FJournal = {Nonlinear Analysis. Real World Applications},
 Journal = {Nonlinear Anal., Real World Appl.},
 ISSN = {1468-1218},
 Volume = {14},
 Number = {5},
 Pages = {2002--2012},
 Year = {2013},
 Language = {English},
 DOI = {10.1016/j.nonrwa.2013.02.004},
 URL = {ddd.uab.cat/record/150623},
 zbMATH = {6298477},
 Zbl = {1293.34047}
}

@Article{HuanYang,
 Author = {Huan, S. M. and Yang, X. S.},
 Title = {On the number of limit cycles in general planar piecewise linear systems},
 FJournal = {Discrete and Continuous Dynamical Systems},
 Journal = {Discrete Contin. Dyn. Syst.},
 ISSN = {1078-0947},
 Volume = {32},
 Number = {6},
 Pages = {2147--2164},
 Year = {2012},
 Language = {English},
 DOI = {10.3934/dcds.2012.32.2147},
 zbMATH = {6039441},
 Zbl = {1248.34033}
}

@Article{Llibre3LC,
 Author = {Llibre, J. and Ponce, E.},
 Title = {Three nested limit cycles in discontinuous piecewise linear differential systems with two zones},
 FJournal = {Dynamics of Continuous, Discrete \& Impulsive Systems. Series B. Applications \& Algorithms},
 Journal = {Dyn. Contin. Discrete Impuls. Syst., Ser. B, Appl. Algorithms},
 ISSN = {1492-8760},
 Volume = {19},
 Number = {3},
 Pages = {325--335},
 Year = {2012},
 Language = {English},
 URL = {online.watsci.org/abstract_pdf/2012v19/v19n3b-pdf/2.pdf},
 zbMATH = {6165831},
 Zbl = {1268.34061}
}

@Article{BragaMello,
 Author = {Braga, D. C. and Mello, L. F.},
 Title = {Limit cycles in a family of discontinuous piecewise linear differential systems with two zones in the plane},
 FJournal = {Nonlinear Dynamics},
 Journal = {Nonlinear Dyn.},
 ISSN = {0924-090X},
 Volume = {73},
 Number = {3},
 Pages = {1283--1288},
 Year = {2013},
 Language = {English},
 DOI = {10.1007/s11071-013-0862-3},
 zbMATH = {6262213},
 Zbl = {1281.34037}
}

@Article{Gasull2020,
 Author = {Gasull, A. and Torregrosa, J. and Zhang, X.},
 Title = {Piecewise linear differential systems with an algebraic line of separation},
 FJournal = {Electronic Journal of Differential Equations (EJDE)},
 Journal = {Electron. J. Differ. Equ.},
 ISSN = {1072-6691},
 Volume = {2020},
 Pages = {14},
 Note = {Id/No 19},
 Year = {2020},
 Language = {English},
 URL = {ejde.math.txstate.edu/Volumes/2020/19/abstr.html#latest},
 zbMATH = {7168196},
 Zbl = {1440.34032}
}

@Article{Han2010,
 Author = {Han, M. and Zhang, W.},
 Title = {On {Hopf} bifurcation in non-smooth planar systems},
 FJournal = {Journal of Differential Equations},
 Journal = {J. Differ. Equations},
 ISSN = {0022-0396},
 Volume = {248},
 Number = {9},
 Pages = {2399--2416},
 Year = {2010},
 Language = {English},
 DOI = {10.1016/j.jde.2009.10.002},
 zbMATH = {5692558},
 Zbl = {1198.34059}
}

@article{carmona2023a,
  author = {Carmona, V. and Fern\'andez-S\'anchez, F. and Novaes, D. D.},
  title = {Uniform upper bound for the number of limit cycles of planar piecewise linear differential systems with two zones separated by a straight line},
  language = {eng},
  format = {article},
  journal = {Applied Mathematics Letters},
  volume = {137},
  pages = {108501},
  year = {2023},
  issn = {18735452, 08939659},
  publisher = {Elsevier Ltd},
  doi = {10.1016/j.aml.2022.108501}
}

@inproceedings{Sotomayor96,
	Author = {Sotomayor, J. and Teixeira, M. A. },
	Booktitle = {Proceedings of the International Conference on Differential Equations, Lisboa},
	Pages = {207--223},
	Title = {Regularization of discontinuous vector fields},
	Year = {1996}
}

@Book{LT,
 Author = {Llibre, J. and Teruel, A. E.},
 Title = {Introduction to the qualitative theory of differential systems. {Planar}, symmetric and continuous piecewise linear systems},
 FSeries = {Birkh{\"a}user Advanced Texts. Basler Lehrb{\"u}cher},
 Series = {Birkh{\"a}user Adv. Texts, Basler Lehrb{\"u}ch.},
 ISSN = {1019-6242},
 ISBN = {978-3-0348-0656-5; 978-3-0348-0657-2},
 Year = {2014},
 Publisher = {Basel: Birkh{\"a}user/Springer},
 Language = {English},
 DOI = {10.1007/978-3-0348-0657-2},
 zbMATH = {6176056},
 Zbl = {1286.34001}
}

@article{carmona2021a,
  author = {Carmona, V. and Fern\'andez-S\'anchez, Fernando and Novaes, D. D.},
  title = {A new simple proof for {L}um–{C}hua's conjecture},
  language = {eng},
  format = {article},
  journal = {Nonlinear Analysis: Hybrid Systems},
  volume = {40},
  pages = {100992},
  year = {2021},
  issn = {1751570x, 18787460},
  publisher = {Elsevier Ltd},
  doi = {10.1016/j.nahs.2020.100992}
}

@article{freire2013a,
  author = {Freire, E. and Ponce, E. and Torres, F.},
  title = {Planar Filippov systems with maximal crossing set and piecewise linear focus dynamics},
  language = {eng},
  format = {article},
  editor = {Ibanez, S, DelRio, JSP, Pumarino, A, Rodriguez, JA},
  journal = {Springer Proceedings in Mathematics and Statistics},
  volume = {54},
  pages = {221-232},
  year = {2013},
  issn = {21941017, 21941009},
  isbn = {3642388299, 3642388302, 9783642388293, 9783642388309},
  publisher = {SPRINGER-VERLAG BERLIN},
  doi = {10.1007/978-3-642-38830-9_13}
}

@article{lum1991a,
  author = {Lum, R. and Chua, L. O.},
  title = {GLOBAL PROPERTIES OF CONTINUOUS PIECEWISE LINEAR VECTOR-FIELDS .1. SIMPLEST CASE IN R2},
  language = {eng},
  format = {article},
  journal = {International Journal of Circuit Theory and Applications},
  volume = {19},
  number = {3},
  pages = {251-307},
  year = {1991},
  issn = {1097007x, 00989886},
  publisher = {JOHN WILEY & SONS LTD},
  doi = {10.1002/cta.4490190305}
}

@Article{HuanYang2,
 Author = {Huan, S. M. and Yang, X. S.},
 Title = {On the number of limit cycles in general planar piecewise linear systems of node-node types},
 FJournal = {Journal of Mathematical Analysis and Applications},
 Journal = {J. Math. Anal. Appl.},
 ISSN = {0022-247X},
 Volume = {411},
 Number = {1},
 Pages = {340--353},
 Year = {2014},
 Language = {English},
 DOI = {10.1016/j.jmaa.2013.08.064},
 zbMATH = {6420024},
 Zbl = {1323.34022}
}

@Article{LiLli,
 Author = {Li, S. and Llibre, J.},
 Title = {Phase portraits of planar piecewise linear refracting systems: focus-saddle case},
 FJournal = {Nonlinear Analysis. Real World Applications},
 Journal = {Nonlinear Anal., Real World Appl.},
 ISSN = {1468-1218},
 Volume = {56},
 Pages = {11},
 Note = {Id/No 103153},
 Year = {2020},
 Language = {English},
 DOI = {10.1016/j.nonrwa.2020.103153},
 URL = {ddd.uab.cat/record/228107},
 zbMATH = {7269768},
 Zbl = {1457.34046}
}

@Article{LiLiuLli,
 Author = {Li, S. and Liu, C. and Llibre, J.},
 Title = {The planar discontinuous piecewise linear refracting systems have at most one limit cycle},
 FJournal = {Nonlinear Analysis. Hybrid Systems},
 Journal = {Nonlinear Anal., Hybrid Syst.},
 ISSN = {1751-570X},
 Volume = {41},
 Pages = {14},
 Note = {Id/No 101045},
 Year = {2021},
 Language = {English},
 DOI = {10.1016/j.nahs.2021.101045},
 zbMATH = {7430099},
 Zbl = {1474.34286}
}

@Article{MedTorr,
 Author = {Medrado, J. C. and Torregrosa, J.},
 Title = {Uniqueness of limit cycles for sewing planar piecewise linear systems},
 FJournal = {Journal of Mathematical Analysis and Applications},
 Journal = {J. Math. Anal. Appl.},
 ISSN = {0022-247X},
 Volume = {431},
 Number = {1},
 Pages = {529--544},
 Year = {2015},
 Language = {English},
 DOI = {10.1016/j.jmaa.2015.05.064},
 URL = {ddd.uab.cat/record/145359},
 zbMATH = {6454289},
 Zbl = {1322.34023}
}

@Article{FrPoTo,
 Author = {Freire, E. and Ponce, E. and Torres, F.},
 Title = {Canonical discontinuous planar piecewise linear systems},
 FJournal = {SIAM Journal on Applied Dynamical Systems},
 Journal = {SIAM J. Appl. Dyn. Syst.},
 ISSN = {1536-0040},
 Volume = {11},
 Number = {1},
 Pages = {181--211},
 Year = {2012},
 Language = {English},
 DOI = {10.1137/11083928X},
 zbMATH = {6025756},
 Zbl = {1242.34020}
}

@article{Otavio23,
 author = {Perez, O. H. and Rond{\'o}n, G. and da Silva, P. R.},
 title = {Slow-fast normal forms arising from piecewise smooth vector fields},
 fjournal = {Journal of Dynamical and Control Systems},
 journal = {J. Dyn. Control Syst.},
 issn = {1079-2724},
 volume = {29},
 number = {4},
 pages = {1709--1726},
 year = {2023},
 language = {English},
 doi = {10.1007/s10883-023-09657-x},
 zbMATH = {7785816},
 Zbl = {1532.34027}
}

@misc{Otavio25,
 author = {De Maesschalck, P. and Huzak, R. and Perez, O. H.},
 title = {Canard cycles of non-linearly regularized piecewise smooth vector fields},
 year = {2025},
 howpublished = {Preprint, {arXiv}:2506.18099 [math.{DS}] (2025)},
 url = {https://arxiv.org/abs/2506.18099},
 arXiv = {arXiv:2506.18099}
}

@book{Jeff,
 author = {Jeffrey, M. R.},
 title = {Hidden dynamics. {The} mathematics of switches, decisions and other discontinuous behaviour},
 isbn = {978-3-030-02106-1; 978-3-030-02107-8},
 year = {2018},
 publisher = {Cham: Springer},
 language = {English},
 doi = {10.1007/978-3-030-02107-8},
 zbMATH = {6967935},
 Zbl = {1406.37003}
}

@article{NovJeff,
 author = {Novaes, D. D. and Jeffrey, M. R.},
 title = {Regularization of hidden dynamics in piecewise smooth flows},
 fjournal = {Journal of Differential Equations},
 journal = {J. Differ. Equations},
 issn = {0022-0396},
 volume = {259},
 number = {9},
 pages = {4615--4633},
 year = {2015},
 language = {English},
 doi = {10.1016/j.jde.2015.06.005},
 zbMATH = {6473158},
 Zbl = {1336.34024}
}

@book{Lovett,
 author = {Lovett, S.},
 title = {Differential geometry of manifolds},
 isbn = {978-1-56881-457-5},
 year = {2010},
 publisher = {Natick, MA: A K Peters},
 language = {English},
 zbMATH = {5776081},
 Zbl = {1205.53001}
}
\end{document}